\documentclass[11pt,a4paper]{article}
\usepackage[T1]{fontenc}
\usepackage{lmodern}
\usepackage{amsmath,amssymb,amsthm,mathtools}
\usepackage{mathrsfs}
\usepackage[textwidth=150mm,textheight=230mm,centering]{geometry}
\usepackage{cite}
\usepackage[hidelinks]{hyperref}
\usepackage{microtype}
\allowdisplaybreaks[2]
\numberwithin{equation}{section}

\newtheorem{theorem}{Theorem}[section]
\newtheorem{proposition}[theorem]{Proposition}
\newtheorem{lemma}[theorem]{Lemma}
\newtheorem{corollary}[theorem]{Corollary}
\newtheorem{maintheorem}{Theorem}

\theoremstyle{definition}
\newtheorem{definition}[theorem]{Definition}
\newtheorem{example}[theorem]{Example}
\newtheorem{remark}[theorem]{Remark}

\newcommand{\R}{\mathbb R}
\newcommand{\C}{\mathbb C}
\newcommand{\bp}{\bar\partial_b}
\newcommand{\Ric}{\operatorname{Ric}}
\newcommand{\tr}{\operatorname{tr}}
\newcommand{\divg}{\operatorname{div}}
\newcommand{\rank}{\operatorname{rank}}
\newcommand{\Sec}{\operatorname{Sec}}
\newcommand{\Bisec}{\operatorname{Bisec}}
\newcommand{\HSC}{\operatorname{HSC}}
\newcommand{\Hess}{\operatorname{Hess}}
\newcommand{\Vol}{\operatorname{Vol}}
\newcommand{\op}{\mathrm{op}}
\newcommand{\HS}{\mathrm{HS}}
\newcommand{\TW}{\mathrm{TW}}
\newcommand{\Rea}{\operatorname{Re}}
\newcommand{\Ima}{\operatorname{Im}}
\newcommand{\Id}{\operatorname{Id}}
\newcommand{\dd}{\,\mathrm{d}v_g}
\newcommand{\ip}[2]{\langle #1,#2\rangle}
\newcommand{\norm}[1]{\lVert #1\rVert}
\newcommand{\Jt}{J_N}
\newcommand{\calR}{\mathcal R}
\newcommand{\calS}{\mathcal S}
\newcommand{\calI}{\mathcal I}
\newcommand{\calL}{\mathcal L}
\newcommand{\calE}{\mathcal E}
\newcommand{\calU}{\mathcal U}
\newcommand{\calQ}{\mathcal Q}
\newcommand{\nablaf}{\nabla^f}
\newcommand{\grad}{\nabla^g}
\newcommand{\eps}{\varepsilon}
\newcommand{\keywords}[1]{\par\smallskip\noindent\textbf{Keywords}\quad #1}
\newcommand{\msc}[1]{\par\smallskip\noindent\textbf{Mathematics Subject Classification (2020)}\quad #1}

\title{Schwarz lemmas and rigidity of \(\bar\partial_b\)-harmonic maps}
\author{Xin Huang \and Hui Liu}
\date{}

\begin{document}
	\maketitle
	
	\begin{abstract}
		We prove Schwarz lemmas for \(\bar\partial_b\)-harmonic maps from complete
		pseudo-Hermitian manifolds into K\"ahler manifolds of negative curvature.
		For maps of bounded generalized dilatation we bound the full and the
		horizontal differential. For CR maps a coupled Bochner identity yields
		rank-dependent horizontal estimates and a bound for the Reeb differential.
		On Sasakian sources and complex hyperbolic targets we classify the maps
		attaining the horizontal bound, bound the deviation from equality in
		integral form by the deficit at one point, and prove compactness of
		almost extremal sequences. As an application, a Siu-type divergence
		identity yields foliation, pluriharmonicity and rank rigidity under
		volume growth conditions.
	\end{abstract}
	\msc{53C43, 32V20, 53C55, 58E20}
	
	\section{Introduction}\label{sec:intro}
	
	The Ahlfors--Schwarz lemma bounds the pull-back metric of a holomorphic
	map from the unit disc into a Riemann surface whose Gaussian curvature is
	bounded above by a negative constant \cite{Ahlfors}. Chern \cite{Chern}
	studied volume ratios of holomorphic maps between Hermitian manifolds of
	equal dimension, and Lu \cite{Lu} derived differential identities for the
	trace of the pull-back Hermitian metric. For a complete K\"ahler source
	with Ricci curvature bounded below by \(-K\), Yau \cite{Yau} proved a
	differential estimate when the holomorphic bisectional curvature of the
	target is bounded above by \(-\kappa<0\). Under an upper bound \(-b<0\)
	for the holomorphic sectional curvature of the target, Royden
	\cite{Royden} obtained the coefficient \(2dK/((d+1)b)\) for the energy
	density of a holomorphic map of rank at most \(d\). Mok and Yau
	\cite{MokYau} proved volume comparison for equidimensional K\"ahler
	manifolds under a lower scalar curvature bound on the complete source and
	a negativity condition on the Ricci form of the target. Shen \cite{Shen}
	established differential estimates for harmonic maps of bounded
	generalized dilatation between complete Riemannian manifolds, under a
	lower Ricci bound on the source and a negative upper bound for the
	sectional curvature of the target. Chen and Nie \cite{ChenNie} showed
	that equality at one point in the estimates of Yau and Royden forces
	global equality, constant rank and total geodesy.
	
	In pseudo-Hermitian geometry, Dong, Ren and Yu \cite{DRY} proved Schwarz
	lemmas for CR and transversally holomorphic maps under curvature bounds.
	Chong, Dong, Ren and Yu \cite{CDRY21} obtained the bound \(K/\kappa\) for
	foliated CR maps from complete pseudo-Hermitian manifolds with
	Tanaka--Webster Ricci curvature bounded below by \(-K\) and torsion
	bounded in \(C^1\) into Hermitian manifolds with bisectional curvature
	bounded above by \(-\kappa<0\). Chong, Dong, Ren and Yang \cite{CDRY19}
	proved foliation and pluriharmonicity results for harmonic and
	pseudoharmonic maps. Huang and Yu \cite{HY} estimated the transverse
	differential of transversally harmonic maps of bounded generalized
	transversal dilatation between Riemannian foliations. Li and Son
	\cite{LS} obtained a CR analogue of Siu's formula. Dong, Liu and Zhao
	\cite{DLZ} introduced a partial energy containing a Reeb term; they
	proved that its critical maps from compact Sasakian manifolds of
	dimension at least five into strongly seminegative K\"ahler manifolds are
	foliated and pluriharmonic, and that they are CR or anti-CR under strong
	negativity and a rank condition. In the present paper we study these maps
	on complete, possibly noncompact, sources. We prove a priori estimates for
	the horizontal and the Reeb differential, analyse the equality case, and
	derive rigidity results.
	
	Let \((M^{2m+1},H,J,\theta)\) be a strictly pseudoconvex pseudo-Hermitian
	manifold with Webster metric \(g\), Levi metric \(G=g|_H\), Reeb field
	\(\xi\), Tanaka--Webster connection \(\nabla\) and torsion \(\mathsf A\),
	normalized by \(d\theta(X,Y)=2g(JX,Y)\). Let \((N,h,J_N)\) be a K\"ahler
	manifold. For \(\Omega\Subset M\) put
	\begin{equation}\label{eq:intro-energy}
		\calE_{\bp,\xi}(f;\Omega)
		=\int_\Omega\Bigl(|\bp f|^2+\frac14|df(\xi)|^2\Bigr)\dd .
	\end{equation}
	A map is called \emph{\(\bp\)-harmonic} if it is critical for
	\eqref{eq:intro-energy} on every relatively compact domain; by
	Proposition~\ref{prop:variation} this is equivalent to
	\begin{equation}\label{eq:intro-pde}
		\tau(f)=2m\,J_N\,df(\xi).
	\end{equation}
	
	\begin{maintheorem}\label{thmA}
		Let \(M\) be complete with \(\Ric^g\ge-K_gg\) and \(|D\xi|_{\op}\le L\),
		where \(D\) is the Levi-Civita connection of \(g\), and let
		\(\Sec_N\le-\kappa<0\). Let \(f\colon M\to N\) be \(\bp\)-harmonic and
		put \(E=|df|^2\), \(U=|df|_H|^2\), with real Hilbert--Schmidt norms. If
		the squared singular values \(\lambda_1\ge\dots\ge\lambda_{2m+1}\ge0\) of
		\(df\) satisfy \(\lambda_1\le\Lambda^2\sum_{j\ge2}\lambda_j\), then
		\[
		E\le\frac{(1+\Lambda^2)(K_g+m^2+2mL)}{\kappa}.
		\]
		Put \(Z=\Delta^D\xi+\Ric_g^\sharp\xi\) and assume in addition
		\(|\mathsf A|_{\HS}\le T\), \(|\ip Z\xi|\le z_0\) and \(|Z_H|\le z_1\).
		If the squared singular values \(\mu_1\ge\dots\ge\mu_{2m}\ge0\) of
		\(df|_H\) satisfy \(\mu_1\le\Lambda_H^2\sum_{j\ge2}\mu_j\), then
		\[
		U\le\max\Bigl\{\frac{2b}{\kappa\delta},
		\frac{a+\sqrt{a^2+2z_1^2}}{2\kappa\delta}\Bigr\},
		\]
		where \(\delta=(1+\Lambda_H^2)^{-1}\), \(a=K_g+3L^2+2mL+4m^2\) and
		\(b=K_g+z_0+2T^2\).
	\end{maintheorem}
	
	\begin{maintheorem}\label{thmB}
		Let \(M\) be complete and connected with \(\Ric^{\TW}\ge-KG\),
		\(|\mathsf A|_{\op}\le a_0\) and \(|\nabla\mathsf A|\le a_1\). Let
		\(f\colon M\to N\) be a CR \(\bp\)-harmonic map, let \(P\) be the
		Hermitian endomorphism of \(H^{1,0}\) defined by
		\(G(PZ,\bar W)=h(df Z,\overline{df W})\), and put
		\(u=\tr P=\tfrac12|df|_H|^2\). Assume \(\rank_\C P\le d\) for an integer
		\(1\le d\le\min\{m,\dim_\C N\}\). If \(\Bisec_N\le-\kappa<0\), then
		\[
		u\le\frac{K+2(m+2)a_0}{\kappa},\qquad
		|df(\xi)|^2\le\frac{2a_0^2}{\kappa}.
		\]
		If in addition \(\dim_\C N\ge2\) and the holomorphic sectional curvature
		satisfies \(\HSC_N\ge-k_{\mathrm h}\), then
		\[
		u\le\frac{K+2(m+2)a_0}{\kappa+\kappa^2/[d(k_{\mathrm h}-\kappa)]}.
		\]
		For the complex hyperbolic space \(\mathbb{CH}^n_\kappa\) of holomorphic
		sectional curvature \(-2\kappa\),
		\[
		u\le\frac{d\,\{K+2(m+2)a_0\}}{(d+1)\kappa}.
		\]
		These horizontal bounds are strict if \(a_0>0\). If \(M\) is Sasakian,
		then \(f\) is foliated and descends locally to a holomorphic map on the
		local K\"ahler quotients of the Reeb foliation.
	\end{maintheorem}
	
	\begin{maintheorem}\label{thmC}
		Let \(M\) be a connected complete Sasakian manifold with
		\(\Ric^{\TW}\ge-KG\), \(K>0\), let \(1\le d\le\min\{m,n\}\), and consider
		CR \(\bp\)-harmonic maps \(f\colon M\to\mathbb{CH}^n_\kappa\) of complex
		rank at most \(d\), with \(P\) and \(u\) as in Theorem~\ref{thmB}. Put
		\(C_d=dK/((d+1)\kappa)\).
		
		If \(u(x_0)=C_d\) for some \(x_0\in M\), then \(u\equiv C_d\) and the
		transverse differential of \(f\) is parallel of complex rank \(d\). Every
		local K\"ahler quotient of the Reeb foliation splits locally as
		\(S\times Q'\) with \(\dim_\C S=d\) and \(\HSC_S=-2K/(d+1)\); the induced
		map is constant along \(Q'\) and restricts to a totally geodesic
		holomorphic homothety on \(S\) with pull-back metric
		\(Kg_S/((d+1)\kappa)\).
		
		For \(x_0\in M\) and \(R>0\) there is \(A_R>0\), depending only on the
		geometry of \(B_{4R}(x_0)\) and on \(K,\kappa,d\), such that every map in
		this class satisfies
		\[
		\int_{B_R(x_0)}\Bigl(|\nabla P|^2+\Bigl|P^2-\frac udP\Bigr|^2\Bigr)\dd
		\le A_R\bigl(C_d-u(x_0)\bigr).
		\]
		Every sequence \(f_\nu\) in this class with \(u_\nu(x_0)\to C_d\) has,
		after composition with holomorphic isometries of the target, a
		subsequence converging in \(C^\infty_{\mathrm{loc}}\) to a map attaining
		equality.
	\end{maintheorem}
	
	\begin{maintheorem}\label{thmD}
		Let \(M^{2m+1}\), \(m\ge2\), be a connected complete Sasakian manifold
		with \(\Ric^g\ge-K_gg\), and let \(N\) be a strongly seminegative K\"ahler
		manifold with \(\Sec_N\le-\kappa<0\). If \(f\colon M\to N\) is
		\(\bp\)-harmonic of bounded generalized dilatation and
		\[
		\Vol B_R(o)\le CR^2\log R\,\log\log R\qquad(R\ge R_0),
		\]
		then \(f\) is foliated, \(\bp\)-pluriharmonic and harmonic.
	\end{maintheorem}
	
	The paper is organized as follows. Section~\ref{sec:geometry} fixes the
	normalizations of the pseudo-Hermitian structure, derives the
	Euler--Lagrange equation \eqref{eq:intro-pde} and develops the algebraic
	curvature estimates on the K\"ahler target. In particular,
	Lemma~\ref{lem:pinched-curvature} converts a pinching of the holomorphic
	sectional curvature into a gain that depends on the rank; this gain is
	the source of all rank-dependent constants below. Section~\ref{sec:bochner}
	collects the differential identities on which the rest of the paper
	rests. Subtracting the Bochner formula for the Reeb energy from the one
	for the full energy removes the indefinite magnetic term and yields an
	identity for the horizontal energy. For CR maps, the horizontal Bochner
	formula and a Jacobi identity along Reeb trajectories combine, through
	the anisotropic operator \(\Delta_b+t\xi^2\), into an exact identity
	whose remainder is a sum of nonnegative geometric quantities. On
	Sasakian sources this identity becomes independent of \(t\). For general
	maps on Sasakian sources, a Siu-type current is constructed so that its
	divergence cancels the magnetic term exactly.
	
	Section~\ref{sec:general} derives the Schwarz lemmas from these
	identities. The horizontal identity, combined with nondegeneracy of the
	differential, gives Theorem~\ref{thmA}; the exact identity, combined
	with an anisotropic maximum principle and a convexity argument along
	Reeb orbits, gives Theorem~\ref{thmB}. Section~\ref{sec:equality}
	analyses the case of equality. Here every term of the remainder must
	vanish, which forces a parallel transverse differential and a local
	splitting of the transverse geometry. The same remainder, integrated
	against cutoff functions and combined with a weak Harnack inequality,
	gives the stability estimate and the compactness statement of
	Theorem~\ref{thmC}. Section~\ref{sec:rigidity} returns to general
	\(\bp\)-harmonic maps: the Siu-type identity of
	Section~\ref{sec:bochner} yields a Caccioppoli inequality, the Schwarz
	estimate of Theorem~\ref{thmA} turns volume growth into energy growth,
	and together they give Theorem~\ref{thmD} and a rank rigidity result.
	Appendix~\ref{app:connections} contains the connection and commutator
	calculations used throughout, and Appendix~\ref{app:localization} the
	localization lemmas for the Riemannian and the anisotropic maximum
	principles.
	
	\section{Preliminaries}\label{sec:geometry}
	
	Throughout, \(K,K_g,a_0,a_1,L,T,z_0,z_1\ge0\) and \(\kappa>0\).
	Completeness and geodesic balls \(B_R(x)\) refer to the Webster metric,
	and \(d_{cc}\) denotes the Carnot--Carath\'eodory distance determined by
	\(G\).
	
	\subsection{Pseudo-Hermitian structures}
	
	Let \(M^{2m+1}\), \(m\ge1\), be a connected strictly pseudoconvex CR
	manifold with contact form \(\theta\), horizontal distribution
	\(H=\ker\theta\), complex structure \(J\) on \(H\) extended by
	\(J\xi=0\), and Reeb field \(\xi\). We use the normalization
	\begin{equation}\label{eq:contact-normalization}
		\theta(\xi)=1,\qquad\iota_\xi d\theta=0,\qquad d\theta(X,Y)=2g(JX,Y),
	\end{equation}
	the Levi metric \(G(X,Y)=\tfrac12d\theta(X,JY)\) and the Webster metric
	\(g=G+\theta\otimes\theta\). The Tanaka--Webster connection \(\nabla\)
	\cite{Tanaka,Webster,DT} preserves \(g\), \(J\) and \(\xi\), and its
	torsion is
	\begin{equation}\label{eq:tw-torsion}
		T^\nabla(X,Y)=2g(JX,Y)\xi+\theta(X)\mathsf AY-\theta(Y)\mathsf AX .
	\end{equation}
	The endomorphism \(\mathsf A\) is \(g\)-symmetric, satisfies
	\(\mathsf A\xi=0\) and \(\mathsf AJ=-J\mathsf A\); in particular
	\(\tr\mathsf A=\tr(J\mathsf A)=0\). The structure is Sasakian if and only
	if \(\mathsf A=0\) \cite{BG}.
	
	\begin{lemma}\label{lem:connection-difference}
		The Levi-Civita connection of \(g\) is \(D=\nabla+\mathsf C\), where
		\begin{equation}\label{eq:connection-difference}
			\begin{aligned}
				\mathsf C(X,Y)={}&-\{g(JX,Y)+g(\mathsf AX,Y)\}\xi+\theta(Y)\mathsf AX\\
				&+\theta(X)JY+\theta(Y)JX .
			\end{aligned}
		\end{equation}
		Consequently
		\begin{equation}\label{eq:reeb-geometry}
			D_X\xi=JX+\mathsf AX,\qquad D_\xi\xi=0,\qquad\divg\xi=0,\qquad
			\sum_\mu\mathsf C(e_\mu,e_\mu)=0 .
		\end{equation}
	\end{lemma}
	
	The proof is given in Appendix~\ref{app:connections}.
	
	Let \(e_1,\dots,e_m,Je_1,\dots,Je_m,\xi\) be a real orthonormal frame and
	\(Z_j=(e_j-iJe_j)/\sqrt2\), \(Z_{\bar j}=\overline{Z_j}\). Then
	\(g(Z_i,Z_{\bar j})=\delta_{ij}\) and
	\(d\theta=2i\sum_j\theta^j\wedge\theta^{\bar j}\). The Tanaka--Webster
	Ricci tensor is
	\(\Ric^{\TW}_{i\bar j}=\sum_kg(R^\nabla(Z_k,Z_{\bar k})Z_i,Z_{\bar j})\),
	regarded as a \(J\)-invariant symmetric form on \(H\). For functions,
	\begin{equation}\label{eq:laplacians}
		\Delta_bu=\sum_j(u_{j\bar j}+u_{\bar jj}),\qquad
		\Delta u=\tr_g\Hess u=\Delta_bu+\xi^2u .
	\end{equation}
	If at least one of the smooth functions \(a,b\) has compact support, the
	divergence theorem gives
	\begin{equation}\label{eq:laplacian-symmetry}
		\begin{aligned}
			\int_Ma\,\Delta b\dd&=-\int_M\ip{da}{db}\dd=\int_Mb\,\Delta a\dd,\\
			\int_Ma\,\Delta_bb\dd&=-\int_M\ip{d_Ha}{d_Hb}\dd,\\
			\int_Ma\,\xi^2b\dd&=-\int_M(\xi a)(\xi b)\dd .
		\end{aligned}
	\end{equation}
	
	\begin{lemma}\label{lem:source-reeb}
		Put \(Z=\Delta^D\xi+\Ric_g^\sharp\xi\). Then
		\begin{equation}\label{eq:z-div}
			Z=2\divg\mathsf A,
		\end{equation}
		and for horizontal \(X\),
		\begin{equation}\label{eq:reeb-curvature}
			R^g(X,\xi)\xi=X-2J\mathsf AX-\mathsf A^2X-(\nabla_\xi\mathsf A)X .
		\end{equation}
		In particular \(\ip{R^g(X,\xi)\xi}{X}\ge(1-2a_0-a_0^2-a_1)|X|^2\)
		whenever \(|\mathsf A|_{\op}\le a_0\) and \(|\nabla\mathsf A|\le a_1\).
	\end{lemma}
	
	The proof is given in Appendix~\ref{app:connections}. Since
	\(D\mathsf A-\nabla\mathsf A\) is algebraic in \(J,\theta,\mathsf A\),
	the bounds \(|\mathsf A|_{\op}\le a_0\) and \(|\nabla\mathsf A|\le a_1\)
	permit the choices
	\begin{equation}\label{eq:source-data}
		L=1+a_0,\qquad T=\sqrt{2m}\,a_0,\qquad
		z_0+z_1\le c_m(a_1+a_0+a_0^2),
	\end{equation}
	where \(z_0=\sup_M|\ip Z\xi|\), \(z_1=\sup_M|Z_H|\) and \(c_m\) depends
	only on \(m\).
	
	\subsection{Maps and the Euler--Lagrange equation}
	
	For a smooth map \(f\colon M\to N\) we write
	\[
	F=df,\qquad V=df(\xi),\qquad E=|df|^2,\qquad
	U=|df|_H|^2,\qquad v=|V|^2,\qquad E=U+v .
	\]
	Let \(\nablaf\) be the pull-back of the Levi-Civita connection of \(N\),
	and let \(\widetilde\nabla\) be the connection on \(T^*M\otimes f^*TN\)
	induced by \(D\) and \(\nablaf\). The second fundamental form and its
	Tanaka--Webster analogue are
	\begin{equation}\label{eq:beta}
		S(X,Y)=(\widetilde\nabla_XF)(Y)=\nablaf_Xdf(Y)-df(D_XY),\qquad
		\beta(X,Y)=\nablaf_Ydf(X)-df(\nabla_YX).
	\end{equation}
	By \eqref{eq:reeb-geometry}, \(\tau(f)=\tr_gS=\tr_g\beta\). For a local
	unitary frame \(\{\eta_\alpha\}\) of \(T^{1,0}N\) write
	\[
	df^{1,0}=f^\alpha_i\theta^i\otimes\eta_\alpha
	+f^\alpha_{\bar i}\theta^{\bar i}\otimes\eta_\alpha
	+f^\alpha_0\theta\otimes\eta_\alpha,
	\]
	and put \(p_i=f^\alpha_i\eta_\alpha\), \(q_i=f^\alpha_{\bar i}\eta_\alpha\),
	\(r=f^\alpha_0\eta_\alpha\). Then \(V=r+\bar r\) and \(v=2|r|^2\). Put
	\(u=\sum_i|p_i|^2\), \(\bar u=\sum_i|q_i|^2\) and \(v_c=|r|^2\); here
	\(\bar u\) denotes the antiholomorphic energy density, not the complex
	conjugate of \(u\). The type decomposition
	\[
	df(Z_i)=p_i+\overline{q_i},\qquad
	df(Z_{\bar i})=q_i+\overline{p_i},\qquad
	|df(Z_i)|^2=|p_i|^2+|q_i|^2
	\]
	gives
	\begin{equation}\label{eq:real-complex-energy}
		\begin{aligned}
			U&=\sum_{i=1}^m\bigl(|df(e_i)|^2+|df(Je_i)|^2\bigr)
			=2\sum_{i,\alpha}\bigl(|f_i^\alpha|^2+|f_{\bar i}^\alpha|^2\bigr)
			=2(u+\bar u),\\
			|\bp f|^2&=\sum_i|df^{1,0}(Z_{\bar i})|^2=\bar u,\\
			E&=U+2|r|^2=2(u+\bar u+v_c).
		\end{aligned}
	\end{equation}
	The Hermitian product on \(T^{1,0}N\) is complex linear in the first
	variable. Unqualified norms of real tensors are Hilbert--Schmidt norms,
	and \(|\bp f|\) is the norm of \(\bp f\) as a section of
	\((H^{0,1})^*\otimes f^*T^{1,0}N\). Covariant derivatives
	\(f^\alpha_{AB}=\beta(Z_A,Z_B)^\alpha\) use \(\nabla\) on source indices
	and \(\nablaf\) on target indices. By Appendix~\ref{app:connections},
	\begin{equation}\label{eq:basic-commutators}
		\begin{aligned}
			f^\alpha_{ij}&=f^\alpha_{ji}, &
			f^\alpha_{i\bar j}-f^\alpha_{\bar ji}&=2i\delta_{ij}f^\alpha_0,\\
			f^\alpha_{0i}-f^\alpha_{i0}&=df^{1,0}(\mathsf AZ_i)^\alpha, &
			f^\alpha_{0\bar i}-f^\alpha_{\bar i0}&=df^{1,0}(\mathsf AZ_{\bar i})^\alpha,
		\end{aligned}
	\end{equation}
	and the third derivatives satisfy the Ricci identity
	\begin{equation}\label{eq:ricci-identity}
		f_{ABC}-f_{ACB}=R^N\bigl(df(Z_C),df(Z_B)\bigr)f_A
		-df\bigl(R^\nabla(Z_C,Z_B)Z_A\bigr)+f_{AD}T^D_{BC}.
	\end{equation}
	
	\begin{definition}\label{def:CR}
		A map is \emph{CR} if \(q_i=0\) for all \(i\), \emph{anti-CR} if
		\(p_i=0\) for all \(i\), and \emph{foliated} if \(V=0\). A foliated map is
		\emph{\(\bp\)-pluriharmonic} if \(f^\alpha_{i\bar j}=f^\alpha_{\bar ji}=0\)
		for all indices.
	\end{definition}
	
	For a CR map, write a real horizontal vector as
	\[
	X=\sum_j(a_je_j+b_jJe_j),\qquad x^j=\frac{a_j+ib_j}{\sqrt2}.
	\]
	Extending \(h\) complex bilinearly and using \(f_{\bar j}=0\),
	\begin{align}
		X^{1,0}&=\frac{X-iJX}2=\sum_jx^jZ_j,\notag\\
		G(X^{1,0},\overline{X^{1,0}})&=\sum_j|x^j|^2
		=\frac12\sum_j(a_j^2+b_j^2)=\frac12G(X,X),\notag\\
		df(X^{1,0})&=\sum_jx^jf_j\in T^{1,0}N,\qquad
		df(\overline{X^{1,0}})=\overline{df(X^{1,0})},\notag\\
		df(X)&=df(X^{1,0})+\overline{df(X^{1,0})},\notag\\
		|df(X)|^2&=2h\bigl(df(X^{1,0}),\overline{df(X^{1,0})}\bigr),
		\label{eq:real-complex-vector-norm}
	\end{align}
	since \(h(df(X^{1,0}),df(X^{1,0}))=0\) by type. With
	\(P_{i\bar j}=\ip{f_i}{f_j}\) this gives
	\begin{align}
		|df(X)|^2&=2G(PX^{1,0},\overline{X^{1,0}})
		=2\sum_{i,j}P_{i\bar j}x^i\overline{x^j}\notag\\
		&\le2\lambda_{\max}(P)\sum_j|x^j|^2
		=\lambda_{\max}(P)\,G(X,X)\le u\,G(X,X).\label{eq:cr-metric}
	\end{align}
	In particular \(|df(e_i)|^2=P_{i\bar i}\),
	\(|df|_H|_{\op}^2=\lambda_{\max}(P)\le u\) and \(U=2\tr P=2u\).
	
	\begin{proposition}\label{prop:variation}
		A smooth map is critical for \(\calE_{\bp,\xi}(\,\cdot\,;\Omega)\) with
		respect to variations supported in \(\Omega\), for every relatively
		compact domain \(\Omega\), if and only if
		\begin{equation}\label{eq:partial-pde}
			\tau(f)=2m\Jt V,
		\end{equation}
		equivalently
		\begin{equation}\label{eq:complex-pde}
			2\sum_jf^\alpha_{\bar jj}+f^\alpha_{00}=0 .
		\end{equation}
	\end{proposition}
	
	\begin{proof}
		The complex antilinear part of the horizontal differential is
		\(F^-_H=\tfrac12(F_H+\Jt F_HJ)\), with \(|F^-_H|^2_{\HS}=2|\bp f|^2\).
		Expanding the square,
		\[
		|\bp f|^2=\bar u
		=\frac14\Bigl(U-\sum_{a=1}^{2m}\ip{\Jt df(e_a)}{df(Je_a)}\Bigr),
		\]
		and therefore the integrand of \eqref{eq:intro-energy} equals
		\[
		\bar u+\frac14v=\frac14\Bigl(|df|^2
		-\sum_{a=1}^{2m}\ip{\Jt df(e_a)}{df(Je_a)}\Bigr).
		\]
		Let \(f_s\) be a variation supported in \(\Omega\), with variation field
		\(W\). Since \(\nabla^N\Jt=0\) and \(\nablaf_{\partial_s}df(X)=\nablaf_XW\),
		\[
		\frac{d}{ds}\Big|_{s=0}\calE_{\bp,\xi}(f_s;\Omega)
		=\frac12\int_\Omega\sum_{\mu=0}^{2m}
		\ip{\nablaf_{e_\mu}W}{df(e_\mu)+\Jt df(Je_\mu)}\dd,\qquad e_0=\xi ;
		\]
		the two mixed terms coincide because
		\(\sum_a\ip{\Jt P(e_a)}{Q(Je_a)}\) is symmetric in \(P,Q\) under
		\(e_a\mapsto Je_a\). Integrating by parts, the one-form
		\(X\mapsto\ip W{df(X)}\) contributes \(-\ip W{\tau(f)}\), and the one-form
		\(X\mapsto\ip W{\Jt df(JX)}\) contributes
		\[
		-\sum_\mu\ip W{\Jt S(e_\mu,Je_\mu)}
		-\sum_\mu\ip W{\Jt df\bigl((D_{e_\mu}J)e_\mu\bigr)}.
		\]
		The first sum vanishes because \(S\) is symmetric and \(J\) is skew. By
		\eqref{eq:connection-difference} and \(\nabla J=0\),
		\[
		(D_XJ)Y=\mathsf C(X,JY)-J\mathsf C(X,Y),\qquad
		\sum_\mu(D_{e_\mu}J)e_\mu=-2m\xi-\tr(J\mathsf A)\xi=-2m\xi .
		\]
		Hence
		\[
		\frac{d}{ds}\Big|_{s=0}\calE_{\bp,\xi}(f_s;\Omega)
		=-\frac12\int_\Omega\ip W{\tau(f)-2m\Jt V}\dd,
		\]
		which proves \eqref{eq:partial-pde}. For the \((1,0)\)-part, use
		\[
		\tau(f)^{1,0}=\sum_j(f_{j\bar j}+f_{\bar jj})+f_{00},\qquad
		f_{j\bar j}=f_{\bar jj}+2ir,\qquad (\Jt V)^{1,0}=ir .
		\]
		Then \eqref{eq:partial-pde} reads \(2\sum_jf_{\bar jj}+2imr+f_{00}=2imr\),
		which is \eqref{eq:complex-pde}.
	\end{proof}
	
	For a CR map \(\sum_jf_{\bar jj}=0\), so \eqref{eq:complex-pde} reduces to
	\begin{equation}\label{eq:cr-geodesic}
		f^\alpha_{00}=0,
	\end{equation}
	that is, \(f\) maps Reeb trajectories to geodesics. For a foliated map
	\eqref{eq:partial-pde} is the harmonic map equation. Equation
	\eqref{eq:partial-pde} is a magnetic map equation in the sense of
	\cite{IM14,IM15}.
	
	\subsection{K\"ahler curvature algebra}
	
	We use \(R(X,Y)=D_XD_Y-D_YD_X-D_{[X,Y]}\), so that the sectional
	curvature is \(\ip{R(X,Y)Y}{X}\) on orthonormal pairs. On \(N\) put
	\(R_{\alpha\bar\beta\gamma\bar\delta}
	=h(R(\eta_\alpha,\bar\eta_\beta)\eta_\gamma,\bar\eta_\delta)\). The
	condition \(\Bisec_N\le-\kappa\) means
	\begin{equation}\label{eq:bisectional}
		R(P,\bar P,Q,\bar Q)\le-\kappa|P|^2|Q|^2\qquad(P,Q\in T^{1,0}N),
	\end{equation}
	and for \(P\ne0\) we write \(\HSC_N(P)=R(P,\bar P,P,\bar P)/|P|^4\). For a
	complex matrix \(W=(W^{\alpha\bar\beta})\) put
	\begin{equation}\label{eq:siu-form}
		\calR(W)=-R_{\alpha\bar\beta\gamma\bar\delta}W^{\alpha\bar\beta}
		\overline{W^{\delta\bar\gamma}} .
	\end{equation}
	Following Siu \cite{Siu}, \(N\) is \emph{strongly seminegative} if
	\(\calR(A\otimes\bar B-C\otimes\bar D)\ge0\) for all
	\(A,B,C,D\in T^{1,0}N\), and \emph{strongly negative} if equality holds
	only when \(A\otimes\bar B=C\otimes\bar D\). For real vectors
	\(X=p+\bar p\) and \(Y=q+\bar q\),
	\[
	\ip{R(X,Y)Y}{X}=-\calR(p\otimes\bar q-q\otimes\bar p),
	\]
	so strong seminegativity implies \(\Sec_N\le0\).
	
	\begin{lemma}\label{lem:pinched-curvature}
		Let \(\dim_\C N\ge2\), \(\Bisec_N\le-\kappa<0\) and
		\(\HSC_N\ge-k_{\mathrm h}\). Then \(k_{\mathrm h}\ge2\kappa\) and
		\begin{equation}\label{eq:holomorphic-gap}
			\HSC_N\le-\kappa-\rho,\qquad\rho=\frac{\kappa^2}{k_{\mathrm h}-\kappa}.
		\end{equation}
		For vectors \(p_1,\dots,p_m\in T^{1,0}N\) put
		\begin{equation}\label{eq:gram-matrix}
			P_{i\bar j}=\ip{p_i}{p_j},\qquad u=\tr P,\qquad |P|^2=\tr P^2,\qquad
			\mathscr C=-\sum_{i,j}R(p_i,\bar p_i,p_j,\bar p_j).
		\end{equation}
		Then
		\begin{equation}\label{eq:pinched-contraction}
			\mathscr C\ge\kappa u^2+\rho|P|^2 ,
		\end{equation}
		and if \(\rank P\le d\),
		\begin{equation}\label{eq:rank-curvature}
			\mathscr C\ge\Bigl(\kappa+\frac\rho d\Bigr)u^2,\qquad
			|P|^2-\frac{u^2}d\ge0 .
		\end{equation}
	\end{lemma}
	
	\begin{proof}
		Let \(p,q\in T^{1,0}N\) be orthonormal, let \(x>0\), and put
		\[
		A_\varphi=p+\sqrt x\,e^{i\varphi}q,\qquad
		B_\varphi=p-\sqrt x\,e^{i\varphi}q,\qquad
		\beta_{pq}=R(p,\bar p,q,\bar q).
		\]
		The K\"ahler symmetries give
		\begin{align*}
			R(q,\bar q,p,\bar p)&=\beta_{pq},&
			R(q,\bar p,p,\bar q)&=\beta_{pq},\\
			R(p,\bar q,q,\bar p)&=\beta_{pq},&
			R(p,\bar p,q,\bar q)&=\beta_{pq}.
		\end{align*}
		In the multilinear expansion of
		\(R(A_\varphi,\bar A_\varphi,B_\varphi,\bar B_\varphi)\) only the terms
		with equally many factors \(e^{i\varphi}\) and \(e^{-i\varphi}\) survive
		averaging; the four terms of order \(x\) carry the signs \(+,-,-,+\).
		Hence
		\begin{align*}
			\frac1{2\pi}\int_0^{2\pi}
			R(A_\varphi,\bar A_\varphi,B_\varphi,\bar B_\varphi)\,d\varphi
			&=\HSC_N(p)+x(\beta_{pq}-\beta_{pq}-\beta_{pq}+\beta_{pq})
			+x^2\HSC_N(q)\\
			&=\HSC_N(p)+x^2\HSC_N(q).
		\end{align*}
		Since \(|A_\varphi|^2=|B_\varphi|^2=1+x\), \eqref{eq:bisectional} and
		\(\HSC_N(q)\ge-k_{\mathrm h}\) give
		\begin{align}
			\HSC_N(p)+x^2\HSC_N(q)&\le-\kappa(1+x)^2,\notag\\
			\HSC_N(p)+\kappa&\le-2\kappa x+(k_{\mathrm h}-\kappa)x^2
			\qquad(x>0).\label{eq:polarization-gap}
		\end{align}
		The right-hand side of \eqref{eq:polarization-gap} is negative for small
		\(x\); if \(k_{\mathrm h}\le\kappa\) it tends to \(-\infty\) as
		\(x\to\infty\), which is impossible. Hence \(k_{\mathrm h}>\kappa\), and
		\(x=\kappa/(k_{\mathrm h}-\kappa)\) gives \eqref{eq:holomorphic-gap}.
		Combined with \(\HSC_N(p)\ge-k_{\mathrm h}\),
		\[
		k_{\mathrm h}-\kappa\ge\frac{\kappa^2}{k_{\mathrm h}-\kappa}
		\quad\Longrightarrow\quad(k_{\mathrm h}-\kappa)^2\ge\kappa^2
		\quad\Longrightarrow\quad k_{\mathrm h}\ge2\kappa .
		\]
		
		The spectral theorem supplies \(A\in U(m)\) with
		\(APA^*=\operatorname{diag}(\lambda_1,\dots,\lambda_m)\). Unitarity
		gives \(\sum_iA_i{}^k\overline{A_i{}^\ell}=\delta_{k\ell}\). Put
		\(\widetilde p_i=\sum_kA_i{}^kp_k\). Then
		\begin{align*}
			\widetilde P_{i\bar j}
			&=\sum_{k,\ell}A_i{}^k\overline{A_j{}^\ell}P_{k\bar\ell}
			=(APA^*)_{i\bar j}=\lambda_i\delta_{ij},\\
			\widetilde u&=\tr(APA^*)=u,\qquad
			|\widetilde P|^2=\tr(AP^2A^*)=|P|^2,\\
			\widetilde{\mathscr C}
			&=-\sum_{k,\ell,a,b}
			\Bigl(\sum_iA_i{}^k\overline{A_i{}^\ell}\Bigr)
			\Bigl(\sum_jA_j{}^a\overline{A_j{}^b}\Bigr)
			R(p_k,\bar p_\ell,p_a,\bar p_b)\\
			&=-\sum_{k,a}R(p_k,\bar p_k,p_a,\bar p_a)=\mathscr C .
		\end{align*}
		The matrix \(A\) acts on the list of vectors through the source indices;
		the target frame is not changed. After suppressing tildes,
		\(|p_i|^2=\lambda_i\) and \(\ip{p_i}{p_j}=0\) for \(i\ne j\). The
		diagonal terms are estimated by \eqref{eq:holomorphic-gap} and the
		off-diagonal terms by \eqref{eq:bisectional}:
		\[
		-R(p_i,\bar p_i,p_i,\bar p_i)\ge(\kappa+\rho)\lambda_i^2,\qquad
		-R(p_i,\bar p_i,p_j,\bar p_j)\ge\kappa\lambda_i\lambda_j\quad(i\ne j);
		\]
		both hold trivially if a vector vanishes. Therefore
		\begin{align*}
			\mathscr C
			&=-\sum_iR(p_i,\bar p_i,p_i,\bar p_i)
			-\sum_{i\ne j}R(p_i,\bar p_i,p_j,\bar p_j)\\
			&\ge(\kappa+\rho)\sum_i\lambda_i^2+\kappa\sum_{i\ne j}\lambda_i\lambda_j
			=\kappa u^2+\rho|P|^2 .
		\end{align*}
		If \(\rank P\le d\), list the eigenvalues as \(\lambda_1,\dots,\lambda_d\),
		including zeros if necessary. Then \(\sum_{i=1}^d\lambda_i=u\),
		\(\sum_{i=1}^d\lambda_i^2=|P|^2\), and
		\[
		|P|^2-\frac{u^2}d=\sum_{i=1}^d\Bigl(\lambda_i-\frac ud\Bigr)^2\ge0,
		\]
		which also gives the first inequality in \eqref{eq:rank-curvature}.
	\end{proof}
	
	\begin{corollary}\label{cor:no-hol-curv}
		If \(\dim_\C N\ge2\) and \(\Bisec_N\le-\kappa<0\), then
		\(\HSC_N(p)<-\kappa\) for every nonzero \(p\in T^{1,0}N\).
	\end{corollary}
	
	\begin{proof}
		At a fixed point \(\HSC_N\) is bounded below on the unit sphere of
		\(T^{1,0}N\); apply \eqref{eq:holomorphic-gap} at that point.
	\end{proof}
	
	For the complex hyperbolic space \(\mathbb{CH}^n_\kappa\) of holomorphic
	sectional curvature \(-2\kappa\),
	\begin{equation}\label{eq:complex-hyperbolic-tensor}
		R_{\alpha\bar\beta\gamma\bar\delta}=-\kappa\bigl(h_{\alpha\bar\beta}
		h_{\gamma\bar\delta}+h_{\alpha\bar\delta}h_{\gamma\bar\beta}\bigr).
	\end{equation}
	For nonzero \(p,q\),
	\[
	\frac{R(p,\bar p,q,\bar q)}{|p|^2|q|^2}
	=-\kappa\Bigl(1+\frac{|\ip pq|^2}{|p|^2|q|^2}\Bigr),\qquad
	-2\kappa\le\Bisec(p,q)\le-\kappa,
	\]
	and consequently
	\begin{equation}\label{eq:complex-hyperbolic-contraction}
		\mathscr C=\kappa(u^2+|P|^2),\qquad
		-\sum_iR(p_i,\bar r,r,\bar p_i)
		=\kappa\Bigl(u|r|^2+\sum_i|\ip{p_i}{r}|^2\Bigr).
	\end{equation}
	We fix a parameter \(\rho\ge0\) such that
	\begin{equation}\label{eq:curvature-contraction-assumption}
		\Bisec_N\le-\kappa,\qquad\mathscr C\ge\kappa u^2+\rho|P|^2
	\end{equation}
	for all \(p_1,\dots,p_m\). Admissible choices are \(\rho=0\) in general,
	\(\rho=\kappa^2/(k_{\mathrm h}-\kappa)\) under the hypotheses of
	Lemma~\ref{lem:pinched-curvature}, and \(\rho=\kappa\) for
	\(\mathbb{CH}^n_\kappa\). For an integer \(d\ge1\) put
	\begin{equation}\label{eq:effective-curvature}
		\alpha_d=\kappa+\frac\rho d .
	\end{equation}
	
	\section{Differential identities}\label{sec:bochner}
	
	In this section \(f\) satisfies \(\tau(f)=\mathfrak c\Jt V\) with
	\(\mathfrak c=2m\). Put \(\mathsf B=D\xi=J+\mathsf A\) and
	\[
	\calQ_H=-\sum_{a,b=1}^{2m}\ip{R^N(Fe_a,Fe_b)Fe_b}{Fe_a},\qquad
	\calQ_0=-\sum_{a=1}^{2m}\ip{R^N(V,Fe_a)Fe_a}{V}.
	\]
	
	\subsection{The full energy and the Reeb energy}
	
	\begin{lemma}\label{lem:weitzenbock}
		Let \(\Delta F=\sum_{\mu=0}^{2m}\bigl(\widetilde\nabla_{e_\mu}
		\widetilde\nabla_{e_\mu}F-\widetilde\nabla_{D_{e_\mu}e_\mu}F\bigr)\).
		Then
		\begin{equation}\label{eq:weitzenbock}
			(\Delta F)(X)=\nablaf_X\tau(f)+F(\Ric_g^\sharp X)
			+\sum_\mu R^N(Fe_\mu,FX)Fe_\mu .
		\end{equation}
	\end{lemma}
	
	\begin{proof}
		At a point choose an orthonormal frame with \(De_\mu=0\). The symmetry
		of \(S\) and the Ricci identity give
		\begin{align*}
			(\widetilde\nabla_{e_\mu}S)(e_\mu,X)
			&=(\widetilde\nabla_{e_\mu}S)(X,e_\mu),\\
			(\widetilde\nabla_{e_\mu}S)(X,e_\mu)
			&=(\widetilde\nabla_XS)(e_\mu,e_\mu)+R^N(Fe_\mu,FX)Fe_\mu
			-F\bigl(R^g(e_\mu,X)e_\mu\bigr).
		\end{align*}
		Summing over \(\mu\),
		\begin{align*}
			\sum_\mu(\widetilde\nabla_XS)(e_\mu,e_\mu)
			&=\nablaf_X\Bigl(\sum_\mu S(e_\mu,e_\mu)\Bigr)=\nablaf_X\tau(f),\\
			-\sum_\mu F\bigl(R^g(e_\mu,X)e_\mu\bigr)&=F(\Ric_g^\sharp X),
		\end{align*}
		which gives \eqref{eq:weitzenbock}.
	\end{proof}
	
	\begin{proposition}\label{prop:coupled-real}
		\begin{align}
			\tfrac12\Delta E={}&|S|^2+\sum_{\mu,\nu=0}^{2m}\Ric^g_{\mu\nu}
			\ip{Fe_\mu}{Fe_\nu}+\calQ_H+2\calQ_0\notag\\
			&+\mathfrak c\sum_{a=1}^{2m}\ip{\Jt\nablaf_{e_a}V}{Fe_a}
			+\mathfrak c\ip{\Jt\nablaf_\xi V}{V},\label{eq:full-bochner}\\
			\tfrac12\Delta v={}&|\nablaf V|^2+\calQ_0
			+\mathfrak c\ip{\Jt\nablaf_\xi V}{V}\notag\\
			&+\ip{F(Z)}{V}+2\sum_{\mu=0}^{2m}\ip{S(e_\mu,\mathsf Be_\mu)}{V}.
			\label{eq:reeb-real-bochner}
		\end{align}
	\end{proposition}
	
	\begin{proof}
		By \(\tfrac12\Delta|F|^2=|\widetilde\nabla F|^2+\ip{\Delta F}{F}\),
		Lemma~\ref{lem:weitzenbock} and
		\(\nablaf_X\tau(f)=\mathfrak c\Jt\nablaf_XV\),
		\begin{align*}
			\tfrac12\Delta E&=|S|^2+\sum_\nu\ip{(\Delta F)(e_\nu)}{Fe_\nu},\\
			\sum_\nu\ip{F(\Ric_g^\sharp e_\nu)}{Fe_\nu}
			&=\sum_{\mu,\nu}\Ric^g_{\mu\nu}\ip{Fe_\mu}{Fe_\nu},\\
			\sum_\nu\ip{\nablaf_{e_\nu}\tau(f)}{Fe_\nu}
			&=\mathfrak c\sum_{a=1}^{2m}\ip{\Jt\nablaf_{e_a}V}{Fe_a}
			+\mathfrak c\ip{\Jt\nablaf_\xi V}{V}.
		\end{align*}
		For the curvature trace, the skew-adjointness of \(R^N\) gives
		\begin{align*}
			\sum_{\mu,\nu=0}^{2m}\ip{R^N(Fe_\mu,Fe_\nu)Fe_\mu}{Fe_\nu}
			={}&-\sum_{a,b=1}^{2m}\ip{R^N(Fe_a,Fe_b)Fe_b}{Fe_a}\\
			&-\sum_{a=1}^{2m}\ip{R^N(V,Fe_a)Fe_a}{V}
			-\sum_{a=1}^{2m}\ip{R^N(Fe_a,V)V}{Fe_a}\\
			={}&\calQ_H+\calQ_0+\calQ_0,
		\end{align*}
		the last two sums coming from the ordered pairs \((\xi,e_a)\) and
		\((e_a,\xi)\). This proves \eqref{eq:full-bochner}.
		
		For \eqref{eq:reeb-real-bochner}, note that \(V=F(\xi)\) and
		\[
		\nablaf_XV=S(X,\xi)+F(D_X\xi)=S(X,\xi)+F(\mathsf BX).
		\]
		Differentiating once more in a normal frame and using
		Lemma~\ref{lem:weitzenbock} with \(X=\xi\),
		\begin{align*}
			\Delta^fV&=(\Delta F)(\xi)+2\sum_\mu S(e_\mu,\mathsf Be_\mu)
			+F(\Delta^D\xi)\\
			&=\mathfrak c\Jt\nablaf_\xi V+F(Z)+2\sum_\mu S(e_\mu,\mathsf Be_\mu)
			+\sum_\mu R^N(Fe_\mu,V)Fe_\mu .
		\end{align*}
		Pairing with \(V\) and using
		\(\tfrac12\Delta|V|^2=|\nablaf V|^2+\ip{\Delta^fV}{V}\) gives
		\eqref{eq:reeb-real-bochner}.
	\end{proof}
	
	\subsection{The horizontal energy}
	
	At a point fix an orthonormal basis \(\{e_a\}\) of \(H\) and put
	\[
	\mathsf S_a=S(e_a,\xi),\qquad T_a=F(\mathsf Be_a),\qquad
	|S_{HH}|^2=\sum_{a,b}|S(e_a,e_b)|^2 .
	\]
	Since \(D_\xi\xi=0\), we have \(\nablaf_{e_a}V=\mathsf S_a+T_a\) and
	\(\nablaf_\xi V=S(\xi,\xi)\). Hence
	\begin{equation}\label{eq:hessian-difference}
		\begin{aligned}
			|S|^2-|\nablaf V|^2
			&=|S_{HH}|^2+2\sum_a|\mathsf S_a|^2+|S(\xi,\xi)|^2
			-\sum_a|\mathsf S_a+T_a|^2-|S(\xi,\xi)|^2\\
			&=|S_{HH}|^2+\sum_a|\mathsf S_a|^2
			-2\sum_a\ip{\mathsf S_a}{T_a}-\sum_a|T_a|^2 .
		\end{aligned}
	\end{equation}
	
	\begin{proposition}\label{prop:horizontal-identity}
		\begin{align}
			\tfrac12\Delta U={}&|S_{HH}|^2+\sum_a|\mathsf S_a|^2
			-2\sum_a\ip{\mathsf S_a}{T_a}-\sum_a|T_a|^2\notag\\
			&+\calQ_H+\calQ_0+\sum_{\mu,\nu}\Ric^g_{\mu\nu}\ip{Fe_\mu}{Fe_\nu}
			-\ip{F(Z)}{V}\notag\\
			&+\mathfrak c\sum_a\ip{\Jt(\mathsf S_a+T_a)}{Fe_a}
			-2\sum_a\ip{S(e_a,\mathsf Ae_a)}{V}.\label{eq:horizontal-identity}
		\end{align}
	\end{proposition}
	
	\begin{proof}
		Subtract \eqref{eq:reeb-real-bochner} from \eqref{eq:full-bochner}; the
		term \(\mathfrak c\ip{\Jt\nablaf_\xi V}{V}\) cancels. Since
		\(\mathsf B\xi=0\), \(\mathsf B=J+\mathsf A\) and \(S\) is symmetric,
		\[
		\sum_{\mu=0}^{2m}S(e_\mu,\mathsf Be_\mu)
		=\sum_{a=1}^{2m}S(e_a,Je_a)+\sum_{a=1}^{2m}S(e_a,\mathsf Ae_a)
		=\sum_{a=1}^{2m}S(e_a,\mathsf Ae_a).
		\]
		Inserting \eqref{eq:hessian-difference} gives \eqref{eq:horizontal-identity}.
	\end{proof}
	
	\begin{lemma}\label{lem:horizontal-differential}
		Assume \(\Ric^g\ge-K_gg\), \(|\mathsf B|_{\op}\le L\),
		\(|\mathsf A|_{\HS}\le T\), \(|\ip Z\xi|\le z_0\) and \(|Z_H|\le z_1\).
		Then
		\begin{equation}\label{eq:horizontal-differential}
			\begin{aligned}
				\tfrac12\Delta U
				\ge{}&\calQ_H+\calQ_0-aU-bv-z_1\sqrt{Uv}\\
				&+\tfrac12|S_{HH}|^2+\tfrac14\sum_a|\mathsf S_a|^2,
			\end{aligned}
		\end{equation}
		with \(a=K_g+3L^2+\mathfrak cL+\mathfrak c^2\) and \(b=K_g+z_0+2T^2\).
	\end{lemma}
	
	\begin{proof}
		Since \(\mathsf B\) preserves \(H\), \(\sum_a|T_a|^2\le L^2U\). From
		\(2|\ip{\mathsf S_a}{T_a}|\le\tfrac12|\mathsf S_a|^2+2|T_a|^2\),
		\[
		\sum_a|\mathsf S_a|^2-2\sum_a\ip{\mathsf S_a}{T_a}-\sum_a|T_a|^2
		\ge\tfrac12\sum_a|\mathsf S_a|^2-3L^2U .
		\]
		By the Cauchy--Schwarz and Young inequalities,
		\begin{align*}
			\mathfrak c\sum_a\ip{\Jt(\mathsf S_a+T_a)}{Fe_a}
			&\ge-\mathfrak c\Bigl(\sum_a|\mathsf S_a|^2\Bigr)^{1/2}\sqrt U
			-\mathfrak cLU\\
			&\ge-\tfrac14\sum_a|\mathsf S_a|^2-\mathfrak c^2U-\mathfrak cLU .
		\end{align*}
		Moreover
		\begin{align*}
			2\Bigl|\sum_a\ip{S(e_a,\mathsf Ae_a)}{V}\Bigr|
			&\le2|S_{HH}|\,|\mathsf A|_{\HS}\sqrt v
			\le\tfrac12|S_{HH}|^2+2T^2v,\\
			\sum_{\mu,\nu}\Ric^g_{\mu\nu}\ip{Fe_\mu}{Fe_\nu}&\ge-K_g(U+v),\\
			-\ip{F(Z)}{V}&\ge-z_1\sqrt{Uv}-z_0v .
		\end{align*}
		Adding these estimates to \eqref{eq:horizontal-identity} gives
		\eqref{eq:horizontal-differential}.
	\end{proof}
	
	\subsection{CR maps: the pull-back tensor and the remainder identity}
	
	Let \(f\) be CR and \(\bp\)-harmonic, and put \(c_m=m+2\). By
	\eqref{eq:basic-commutators} and \eqref{eq:cr-geodesic},
	\begin{equation}\label{eq:cr-first-derivatives}
		f_{\bar i}=0,\qquad f_{i\bar j}=2i\delta_{ij}r,\qquad f_{00}=0,\qquad
		f_{0i}=f_{i0};
	\end{equation}
	the last identity holds because \(\mathsf AZ_i\) has type \((0,1)\) and
	\(q_i=0\). The remaining mixed derivative is
	\begin{equation}\label{eq:cr-anti-reeb}
		f_{0\bar i}=df^{1,0}(\mathsf AZ_{\bar i}).
	\end{equation}
	Put
	\[
	P_{i\bar j}=\ip{f_i}{f_j},\qquad s_0=\sum_i|f_{i0}|^2,\qquad
	T_A=\sum_i|df^{1,0}(\mathsf AZ_{\bar i})|^2,
	\]
	and regard \(P\) as the Hermitian endomorphism of \(H^{1,0}\) with
	\(G(PX,\bar Y)=h(dfX,\overline{dfY})\). Derivatives of \(P\) are taken
	with respect to the connection induced by \(\nabla\).
	
	\begin{proposition}\label{prop:cr-tensor}
		Define sections of \(f^*T^{1,0}N\) by
		\begin{align}
			\mathcal V_i&=2ic_mf_{i0}+\sum_j\Ric^{\TW}_{i\bar j}f_j
			-\sum_jR^N(f_j,\bar f_j)f_i,\label{eq:cr-C}\\
			\mathcal J_i&=-R^N(f_i,\bar r)r-df^{1,0}(\mathsf A^2Z_i).\label{eq:cr-H}
		\end{align}
		Then
		\begin{align}
			(\Delta_bP)_{i\bar j}&=2\sum_k\bigl(\ip{f_{ik}}{f_{jk}}
			+\ip{f_{i\bar k}}{f_{j\bar k}}\bigr)+\ip{\mathcal V_i}{f_j}
			+\ip{f_i}{\mathcal V_j},\label{eq:tensor-horizontal}\\
			(\nabla_\xi^2P)_{i\bar j}&=2\ip{f_{i0}}{f_{j0}}+\ip{\mathcal J_i}{f_j}
			+\ip{f_i}{\mathcal J_j}.\label{eq:tensor-reeb}
		\end{align}
	\end{proposition}
	
	\begin{proof}
		Choose unitary frames whose connection forms vanish at the point.
		Metric compatibility and the convention on the Hermitian product give
		\begin{align*}
			P_{i\bar j,k}&=\ip{f_{ik}}{f_j}+\ip{f_i}{f_{j\bar k}},\\
			P_{i\bar j,k\bar k}&=\ip{f_{ik\bar k}}{f_j}+\ip{f_{ik}}{f_{jk}}
			+\ip{f_{i\bar k}}{f_{j\bar k}}+\ip{f_i}{f_{j\bar kk}},\\
			P_{i\bar j,\bar kk}&=\ip{f_{i\bar kk}}{f_j}+\ip{f_{i\bar k}}{f_{j\bar k}}
			+\ip{f_{ik}}{f_{jk}}+\ip{f_i}{f_{jk\bar k}}.
		\end{align*}
		Differentiating \(f_{i\bar k}=2i\delta_{ik}r\) and using
		\(f_{ik}=f_{ki}\) together with \eqref{eq:ricci-identity},
		\begin{align*}
			\sum_kf_{i\bar kk}&=2if_{0i},\qquad
			\sum_kf_{k\bar ki}=2imf_{0i},\\
			\sum_kf_{ik\bar k}&=\sum_kf_{k\bar ki}+2if_{i0}
			+\sum_j\Ric^{\TW}_{i\bar j}f_j-\sum_kR^N(f_k,\bar f_k)f_i .
		\end{align*}
		Here the torsion term of \eqref{eq:ricci-identity} gives \(2if_{i0}\),
		the source term gives
		\(-\sum_kdf^{1,0}(R^\nabla(Z_{\bar k},Z_i)Z_k)
		=\sum_j\Ric^{\TW}_{i\bar j}f_j\), and the target term was rewritten by
		\[
		R^N(\bar f_k,f_i)f_k=-R^N(f_i,\bar f_k)f_k=-R^N(f_k,\bar f_k)f_i .
		\]
		With \(f_{0i}=f_{i0}\) we obtain
		\begin{equation}\label{eq:cr-contracted-commutation}
			\sum_k(f_{ik\bar k}+f_{i\bar kk})=2ic_mf_{i0}
			+\sum_j\Ric^{\TW}_{i\bar j}f_j-\sum_kR^N(f_k,\bar f_k)f_i=\mathcal V_i .
		\end{equation}
		Adding \(P_{i\bar j,k\bar k}\) and \(P_{i\bar j,\bar kk}\) and summing
		over \(k\) proves \eqref{eq:tensor-horizontal}.
		
		Along a Reeb trajectory choose a \(\nabla\)-parallel source frame.
		Since \(\nabla\xi=0\), \eqref{eq:tw-torsion} gives
		\[
		[\xi,Z_i]=-\mathsf AZ_i,\qquad \nablaf_\xi r=f_{00}=0,\qquad
		\nablaf_\xi f_i=\nablaf_{Z_i}r .
		\]
		The curvature identity for the pull-back connection yields
		\begin{align*}
			\nablaf_\xi\nablaf_\xi f_i
			&=\nablaf_\xi\nablaf_{Z_i}r\\
			&=\nablaf_{Z_i}\nablaf_\xi r+R^N(df(\xi),df(Z_i))r
			+\nablaf_{[\xi,Z_i]}r\\
			&=-R^N(f_i,\bar r)r-\nablaf_{\mathsf AZ_i}r\\
			&=-R^N(f_i,\bar r)r-df^{1,0}(\mathsf A^2Z_i)=\mathcal J_i,
		\end{align*}
		where the last step uses \eqref{eq:cr-anti-reeb}, since
		\(\mathsf AZ_i\) has type \((0,1)\). Differentiating \(P_{i\bar j}\)
		twice along \(\xi\) gives \eqref{eq:tensor-reeb}.
	\end{proof}
	
	\begin{corollary}\label{cor:cr-horizontal}
		With \(\mathscr C\) defined by \eqref{eq:gram-matrix} for \(p_i=f_i\) and
		\(\mathscr C_0=-\sum_iR(f_i,\bar r,r,\bar f_i)\),
		\begin{align}
			\Delta_bu={}&2\sum_{i,j}\bigl(|f_{ij}|^2+|f_{i\bar j}|^2\bigr)
			+2\sum_{i,j}\Ric^{\TW}_{i\bar j}\ip{f_j}{f_i}\notag\\
			&+2\mathscr C+4c_m\Rea\Bigl(i\sum_i\ip{f_{i0}}{f_i}\Bigr),
			\label{eq:cr-horizontal-bochner}\\
			\xi^2u={}&2s_0+2\mathscr C_0-2T_A .\label{eq:cr-reeb-jacobi}
		\end{align}
		If \(\Bisec_N\le-\kappa\) and \(|\mathsf A|_{\op}\le a_0\), then
		\begin{equation}\label{eq:cr-reeb-lower}
			\xi^2u\ge2s_0+2\kappa uv_c-2a_0^2u .
		\end{equation}
	\end{corollary}
	
	\begin{proof}
		Taking the trace of \eqref{eq:tensor-horizontal} gives
		\eqref{eq:cr-horizontal-bochner}. For the trace of
		\eqref{eq:tensor-reeb}, write \(\mathsf AZ_{\bar i}=\sum_jA_{ji}Z_j\).
		The symmetry of \(\mathsf A\) gives
		\begin{align*}
			\Rea\sum_i\ip{df^{1,0}(\mathsf A^2Z_i)}{f_i}
			&=\sum_i\Bigl|\sum_jA_{ji}f_j\Bigr|^2=T_A,\qquad 0\le T_A\le a_0^2u,\\
			\mathscr C_0&=-\sum_iR(r,\bar r,f_i,\bar f_i)\ge\kappa uv_c,
		\end{align*}
		the last inequality by \eqref{eq:bisectional}. This proves
		\eqref{eq:cr-reeb-jacobi} and \eqref{eq:cr-reeb-lower}.
	\end{proof}
	
	Assume \(\Ric^{\TW}\ge-KG\), \(|\mathsf A|_{\op}\le a_0\),
	\eqref{eq:curvature-contraction-assumption} and \(\rank_\C P\le d\).
	Define
	\begin{equation}\label{eq:remainder-components}
		\begin{aligned}
			D_R&=\sum_{i,j}(\Ric^{\TW}_{i\bar j}+K\delta_{ij})\ip{f_j}{f_i}, &
			D_N&=\mathscr C-\kappa u^2-\rho|P|^2,\\
			D_0&=\mathscr C_0-\kappa uv_c, &
			D_A&=a_0^2u-T_A,\qquad D_P=|P|^2-\frac{u^2}d .
		\end{aligned}
	\end{equation}
	The source Ricci bound, \eqref{eq:curvature-contraction-assumption},
	Corollary~\ref{cor:cr-horizontal} and Lemma~\ref{lem:pinched-curvature}
	give \(D_R,D_N,D_0,D_A,D_P\ge0\), and
	\[
	\mathscr C=\alpha_du^2+D_N+\rho D_P .
	\]
	For \(t>0\) put
	\begin{equation}\label{eq:Bt-Ct}
		b_t=K+ta_0^2+\frac{c_m^2}t,\qquad C_t=\frac{b_t}{\alpha_d},\qquad
		\calL_t=\Delta_b+t\xi^2,
	\end{equation}
	and let \(\Gamma_t(a,b)=\ip{d_Ha}{d_Hb}+t(\xi a)(\xi b)\),
	\(\Gamma_t(a)=\Gamma_t(a,a)\). In a target unitary frame the curvature
	contractions read
	\begin{align}
		\mathscr C&=-\sum_{i,j,\alpha,\beta,\gamma,\delta}
		R_{\alpha\bar\beta\gamma\bar\delta}f_i^\alpha\overline{f_i^\beta}
		f_j^\gamma\overline{f_j^\delta},\label{eq:curvature-component-trace}\\
		\mathscr C_0&=-\sum_{i,\alpha,\beta,\gamma,\delta}
		R_{\alpha\bar\beta\gamma\bar\delta}f_i^\alpha\overline{r^\beta}
		r^\gamma\overline{f_i^\delta}.\label{eq:curvature-component-reeb}
	\end{align}
	
	\begin{proposition}\label{prop:exact-remainder}
		\begin{equation}\label{eq:exact-remainder}
			\calL_tu=2\alpha_du(u-C_t)+\mathscr D_t,
		\end{equation}
		where
		\begin{equation}\label{eq:Dt-definition}
			\begin{aligned}
				\mathscr D_t={}&2\sum_{i,j}|f_{ij}|^2+8mv_c+2t\kappa uv_c
				+2t\sum_i\Bigl|f_{i0}-\frac{ic_m}tf_i\Bigr|^2\\
				&+2D_R+2D_N+2\rho D_P+2tD_0+2tD_A\ \ge0 .
			\end{aligned}
		\end{equation}
	\end{proposition}
	
	\begin{proof}
		By \eqref{eq:cr-first-derivatives},
		\(\sum_{i,j}|f_{i\bar j}|^2=\sum_i|2ir|^2=4mv_c\). Inserting
		\eqref{eq:remainder-components} into \eqref{eq:cr-horizontal-bochner}
		and \eqref{eq:cr-reeb-jacobi},
		\begin{align*}
			\calL_tu={}&2\sum_{i,j}|f_{ij}|^2+8mv_c-2Ku+2D_R
			+2\kappa u^2+2\rho|P|^2+2D_N\\
			&+2t\kappa uv_c+2tD_0-2ta_0^2u+2tD_A\\
			&+2ts_0+4c_m\Rea\Bigl(i\sum_i\ip{f_{i0}}{f_i}\Bigr).
		\end{align*}
		For a Hermitian product linear in the first variable and \(z\in\C\),
		\begin{align*}
			|a-zb|^2&=|a|^2+|z|^2|b|^2-2\Rea\bigl(\bar z\ip ab\bigr),\\
			-2\Rea\Bigl(\overline{\tfrac{ic_m}t}\ip{f_{i0}}{f_i}\Bigr)
			&=\frac{2c_m}t\Rea\bigl(i\ip{f_{i0}}{f_i}\bigr).
		\end{align*}
		Hence
		\begin{align*}
			2t\sum_i\Bigl|f_{i0}-\frac{ic_m}tf_i\Bigr|^2
			&=2ts_0+\frac{2c_m^2}tu+4c_m\Rea\Bigl(i\sum_i\ip{f_{i0}}{f_i}\Bigr),\\
			2\kappa u^2+2\rho|P|^2&=2\alpha_du^2+2\rho D_P .
		\end{align*}
		Substituting these identities, the linear terms in \(u\) combine to
		\(-2(K+ta_0^2+c_m^2/t)u=-2\alpha_dC_tu\), which proves
		\eqref{eq:exact-remainder}.
	\end{proof}
	
	\paragraph{Sasakian sources.}
	Let \(M\) be Sasakian and \(f\) foliated; the latter follows from
	Theorem~\ref{thm:reeb-bound} below. Then \(a_0=0\), \(r=0\), and
	\eqref{eq:cr-first-derivatives} gives
	\[
	f_{i\bar j}=f_{i0}=0,\qquad\xi u=0,\qquad\calL_tu=\Delta u=\Delta_bu .
	\]
	In \eqref{eq:Dt-definition} the following terms vanish or simplify:
	\begin{align*}
		8mv_c&=0,& 2t\kappa uv_c&=0,& D_0&=0,& D_A&=0,\\
		C_t&=C_0+\frac{c_m^2}{\alpha_dt},&
		2t\sum_i\Bigl|f_{i0}-\frac{ic_m}tf_i\Bigr|^2&=\frac{2c_m^2}tu,
	\end{align*}
	where \(C_0=K/\alpha_d\). Consequently
	\(2\alpha_du(u-C_t)+\mathscr D_t=2\alpha_du(u-C_0)+\mathscr D_{\mathrm S}\),
	and \eqref{eq:exact-remainder} becomes, independently of \(t\),
	\begin{equation}\label{eq:sasaki-deficit-identity}
		\begin{gathered}
			\Delta u=2\alpha_du(u-C_0)+\mathscr D_{\mathrm S},\qquad
			\Delta w=2\alpha_duw-\mathscr D_{\mathrm S},\\
			w=C_0-u,\qquad
			\mathscr D_{\mathrm S}=2\sum_{i,j}|f_{ij}|^2+2D_R+2D_N+2\rho D_P\ge0 .
		\end{gathered}
	\end{equation}
	
	\subsection{A Siu-type divergence identity}\label{sec:siu}
	
	In this subsection \(M\) is Sasakian with \(m\ge2\), \(N\) is strongly
	seminegative, and \(f\) is \(\bp\)-harmonic; \(f\) is not assumed to be
	CR or foliated. Put
	\begin{equation}\label{eq:siu-derivative-definitions}
		\mathsf T_{i\bar j}=f_{i\bar j},\qquad\sigma=\sum_if_{\bar ii},\qquad
		\mathsf Y=f_{00}=\nablaf_\xi r,\qquad\mathsf Y_i=\nablaf_{Z_i}\mathsf Y .
	\end{equation}
	By \eqref{eq:complex-pde} and \eqref{eq:basic-commutators},
	\begin{equation}\label{eq:siu-traces}
		\sigma=-\tfrac12\mathsf Y,\qquad\tr\mathsf T=\sigma+2imr .
	\end{equation}
	Define
	\begin{equation}\label{eq:siu-B}
		\mathsf T^\circ_{i\bar j}=\mathsf T_{i\bar j}
		-\frac{\tr\mathsf T}m\delta_{ij},\qquad
		\calU_{ij}=q_j\otimes\bar p_i-q_i\otimes\bar p_j,\qquad
		\calS=\frac12\sum_{i,j}\calR(\calU_{ij})\ge0 .
	\end{equation}
	Expanding \eqref{eq:siu-form} and exchanging \(i,j\) in one pair of
	terms,
	\begin{equation}\label{eq:siu-expanded}
		\calS=\Rea\sum_{i,j}\bigl\{R(q_j,\bar p_i,p_j,\bar q_i)
		-R(q_j,\bar p_i,p_i,\bar q_j)\bigr\}.
	\end{equation}
	For a horizontal \((0,1)\)-form \(\eta=\eta_{\bar j}\theta^{\bar j}\)
	define the real vector field
	\begin{equation}\label{eq:real-current}
		X_\eta=\tfrac12\sum_j\bigl(\eta_{\bar j}Z_j
		+\overline{\eta_{\bar j}}Z_{\bar j}\bigr).
	\end{equation}
	Then \(\ip{d\phi}{X_\eta}=\Rea\sum_j\phi_j\eta_{\bar j}\),
	\(|X_\eta|^2=\tfrac12|\eta|^2\), and, since the trace of \(\mathsf C\)
	vanishes,
	\begin{equation}\label{eq:current-divergence}
		\divg X_\eta=\Rea\sum_j\eta_{\bar j,j}.
	\end{equation}
	Put
	\[
	\eta_{\bar j}=\sum_i\ip{\mathsf T^\circ_{i\bar j}}{p_i},\qquad
	\zeta_{\bar i}=\ip{\mathsf Y}{p_i},\qquad a_m=\frac{m-1}{2m},\qquad
	T_0=\Rea\ip{i\mathsf Y}{r}.
	\]
	
	\begin{lemma}\label{lem:siu-currents}
		\begin{align}
			\divg X_\eta&=|\mathsf T^\circ|^2+\calS
			-a_m\Rea\sum_i\ip{\mathsf Y_i}{p_i},\label{eq:eta-div}\\
			\divg X_\zeta&=\Rea\sum_i\ip{\mathsf Y_i}{p_i}-\tfrac12|\mathsf Y|^2
			-2mT_0,\label{eq:zeta-div}
		\end{align}
		and consequently
		\begin{equation}\label{eq:siu-combined-current}
			\divg(X_\eta+a_mX_\zeta)=|\mathsf T^\circ|^2+\calS
			-\frac{a_m}2|\mathsf Y|^2-(m-1)T_0 .
		\end{equation}
	\end{lemma}
	
	\begin{proof}
		In unitary frames normal at the point, since \(\mathsf T^\circ\) is trace
		free,
		\[
		\Rea\sum_j\eta_{\bar j,j}=|\mathsf T^\circ|^2
		+\Rea\sum_{i,j}\ip{\mathsf T^\circ_{i\bar j,j}}{p_i}.
		\]
		Write \(r_j=f_{0j}\). Differentiating
		\(f_{i\bar j}-f_{\bar ji}=2i\delta_{ij}r\) gives
		\(f_{i\bar jj}=f_{\bar jij}+2i\delta_{ij}r_j\). On a Sasakian manifold
		\(T^\nabla(Z_i,Z_j)=0\) and \(R^\nabla(Z_j,Z_i)=0\), the transverse
		curvature being of type \((1,1)\); hence \eqref{eq:ricci-identity} with
		\((A,B,C)=(\bar j,i,j)\) contains only the target term, and since
		\(df(Z_j)=p_j+\bar q_j\),
		\begin{align*}
			f_{\bar jij}-f_{\bar jji}
			&=R^N(df(Z_j),df(Z_i))q_j
			=R^N(p_j,\bar q_i)q_j-R^N(p_i,\bar q_j)q_j,\\
			\sum_jf_{\bar jij}&=\sigma_i+\mathcal A_i,\qquad
			\mathcal A_i=\sum_j\bigl\{R^N(p_j,\bar q_i)q_j-R^N(p_i,\bar q_j)q_j\bigr\}.
		\end{align*}
		Therefore
		\[
		\sum_j\mathsf T^\circ_{i\bar j,j}
		=\sigma_i+2ir_i-\frac1m(\sigma_i+2imr_i)+\mathcal A_i
		=\frac{m-1}m\sigma_i+\mathcal A_i=-a_m\mathsf Y_i+\mathcal A_i,
		\]
		where \eqref{eq:siu-traces} was used. The K\"ahler symmetries give
		\begin{align*}
			\Rea\sum_i\ip{\mathcal A_i}{p_i}
			&=\Rea\sum_{i,j}\bigl\{R(p_j,\bar q_i,q_j,\bar p_i)
			-R(p_i,\bar q_j,q_j,\bar p_i)\bigr\}\\
			&=\Rea\sum_{i,j}\bigl\{R(q_j,\bar p_i,p_j,\bar q_i)
			-R(q_j,\bar p_i,p_i,\bar q_j)\bigr\}=\calS ,
		\end{align*}
		which proves \eqref{eq:eta-div}. Next,
		\[
		\sum_i\zeta_{\bar i,i}=\sum_i\ip{\mathsf Y_i}{p_i}
		+\ip{\mathsf Y}{\tr\mathsf T}
		=\sum_i\ip{\mathsf Y_i}{p_i}-\tfrac12|\mathsf Y|^2-2im\ip{\mathsf Y}{r},
		\]
		and taking real parts gives \eqref{eq:zeta-div}. In the combination
		\eqref{eq:siu-combined-current} the terms
		\(\Rea\sum_i\ip{\mathsf Y_i}{p_i}\) cancel.
	\end{proof}
	
	Put \(s=\sum_i(|f_{0i}|^2+|f_{0\bar i}|^2)\), \(\omega_0=|\mathsf Y|^2\)
	and \(b_m=4m/(m-1)\).
	
	\begin{proposition}\label{prop:positive-divergence}
		The vector field \(\mathcal Y=\grad v_c+b_m(X_\eta+a_mX_\zeta)\)
		satisfies
		\begin{equation}\label{eq:positive-divergence}
			\divg\mathcal Y=\calI:=2s+\omega_0+b_m\bigl(|\mathsf T^\circ|^2
			+\calS\bigr)+\calQ_0\ \ge0
		\end{equation}
		and
		\begin{equation}\label{eq:Y-control}
			|\mathcal Y|\le c_m^{\mathrm S}\sqrt E\sqrt\calI,\qquad
			c_m^{\mathrm S}=1+\sqrt2+\sqrt{\frac{2m}{m-1}} .
		\end{equation}
	\end{proposition}
	
	\begin{proof}
		For \(\mathsf A=0\) we have \(\mathsf B=J\), \(Z=0\) by \eqref{eq:z-div},
		and \(\sum_\mu S(e_\mu,Je_\mu)=0\), so \eqref{eq:reeb-real-bochner}
		becomes \(\tfrac12\Delta v=|\nablaf V|^2+2m\ip{\Jt\nablaf_\xi V}{V}
		+\calQ_0\). With \(v=2v_c\), \(|\nablaf V|^2=2(s+\omega_0)\) and
		\(\ip{\Jt\nablaf_\xi V}{V}=2T_0\),
		\begin{equation}\label{eq:reeb-siu-local}
			\Delta v_c=2s+2\omega_0+4mT_0+\calQ_0 .
		\end{equation}
		The coefficients satisfy
		\[
		b_m(m-1)=4m,\qquad \frac{b_ma_m}2=\frac{4m}{m-1}\cdot\frac{m-1}{2m}
		\cdot\frac12=1 .
		\]
		Adding \(b_m\) times \eqref{eq:siu-combined-current} to
		\eqref{eq:reeb-siu-local},
		\begin{align*}
			\divg\mathcal Y&=\Delta v_c+b_m\divg(X_\eta+a_mX_\zeta)\\
			&=2s+2\omega_0+4mT_0+\calQ_0+b_m\bigl(|\mathsf T^\circ|^2+\calS\bigr)
			-\frac{b_ma_m}2\omega_0-b_m(m-1)T_0\\
			&=2s+\omega_0+b_m\bigl(|\mathsf T^\circ|^2+\calS\bigr)+\calQ_0 .
		\end{align*}
		Strong seminegativity gives \(\calS\ge0\) and \(\Sec_N\le0\), hence
		\(\calQ_0\ge0\). This is \eqref{eq:positive-divergence}.
		
		For \eqref{eq:Y-control}, metric compatibility gives
		\(dv_c(Y)=2\Rea\ip{\nablaf_Yr}{r}\). With
		\(e_i=(Z_i+Z_{\bar i})/\sqrt2\) and \(Je_i=i(Z_i-Z_{\bar i})/\sqrt2\),
		\begin{align*}
			\nablaf_{e_i}r&=\frac{f_{0i}+f_{0\bar i}}{\sqrt2},\qquad
			\nablaf_{Je_i}r=\frac{i(f_{0i}-f_{0\bar i})}{\sqrt2},\qquad
			\nablaf_\xi r=f_{00}=\mathsf Y,\\
			|\nablaf r|^2&=\frac12\sum_i\bigl(|f_{0i}+f_{0\bar i}|^2
			+|f_{0i}-f_{0\bar i}|^2\bigr)+|\mathsf Y|^2=s+\omega_0,\\
			|dv_c|&\le2|r|\,|\nablaf r|=2|r|\sqrt{s+\omega_0}.
		\end{align*}
		By Cauchy--Schwarz,
		\begin{align*}
			|\eta|^2&=\sum_j\Bigl|\sum_i\ip{\mathsf T^\circ_{i\bar j}}{p_i}\Bigr|^2
			\le\sum_j\Bigl(\sum_i|\mathsf T^\circ_{i\bar j}|^2\Bigr)
			\Bigl(\sum_i|p_i|^2\Bigr)=u|\mathsf T^\circ|^2,\\
			|\zeta|^2&=\sum_i|\ip{\mathsf Y}{p_i}|^2\le u|\mathsf Y|^2,\qquad
			|X_\eta|^2\le\tfrac12u|\mathsf T^\circ|^2,\qquad
			|X_\zeta|^2\le\tfrac12u|\mathsf Y|^2 .
		\end{align*}
		Using \(2u\le E\), \(2|r|^2\le E\), \(s+\omega_0\le\calI\),
		\(b_m|\mathsf T^\circ|^2\le\calI\), \(|\mathsf Y|^2\le\calI\) and
		\(b_ma_m=2\),
		\begin{align*}
			|dv_c|&\le2\sqrt{E/2}\,\sqrt\calI=\sqrt2\sqrt E\sqrt\calI,\\
			b_m|X_\eta|&\le\sqrt{b_m}\sqrt{u/2}\sqrt{b_m|\mathsf T^\circ|^2}
			\le\sqrt{b_m/2}\,\sqrt E\sqrt\calI=\sqrt{\tfrac{2m}{m-1}}\sqrt E\sqrt\calI,\\
			b_ma_m|X_\zeta|&=2|X_\zeta|\le\sqrt{2u}\,|\mathsf Y|\le\sqrt E\sqrt\calI .
		\end{align*}
		Summing gives \(|\mathcal Y|\le|dv_c|+b_m|X_\eta|+b_ma_m|X_\zeta|
		\le c_m^{\mathrm S}\sqrt E\sqrt\calI\).
	\end{proof}
	
	\section{Schwarz lemmas}\label{sec:general}
	
	\subsection{General maps}
	
	\begin{definition}\label{def:dilatation}
		A linear map \(F\) between Euclidean spaces with squared singular values
		\(\lambda_1\ge\dots\ge\lambda_d\ge0\) has \emph{generalized dilatation at
			most \(\Lambda\)} if \(\lambda_1\le\Lambda^2\sum_{j\ge2}\lambda_j\). The
		full and horizontal dilatation of \(f\) refer to \(df\) and \(df|_H\).
	\end{definition}
	
	\begin{lemma}\label{lem:linear-algebra}
		Let \(F\) have generalized dilatation at most \(\Lambda\), and put
		\(s=|F|^2\), \(\delta=(1+\Lambda^2)^{-1}\). Then for every \(Y\),
		\begin{align}
			\lambda_1&\le(1-\delta)s,\label{eq:lambda-cap}\\
			s^2-\tr(F^*F)^2&\ge\delta s^2,\label{eq:wedge-lower}\\
			s|Y|^2-|F^*Y|^2&\ge\delta s|Y|^2 .\label{eq:mixed-lower}
		\end{align}
	\end{lemma}
	
	\begin{proof}
		Adding \(\Lambda^2\lambda_1\) to both sides of the hypothesis gives
		\((1+\Lambda^2)\lambda_1\le\Lambda^2s\), which is \eqref{eq:lambda-cap}.
		Moreover
		\begin{align*}
			\tr(F^*F)^2&=\sum_j\lambda_j^2\le\lambda_1s,&
			s^2-\tr(F^*F)^2&\ge s(s-\lambda_1)\ge\delta s^2,\\
			|F^*Y|^2&=\ip{FF^*Y}{Y}\le\lambda_1|Y|^2,&
			s|Y|^2-|F^*Y|^2&\ge(s-\lambda_1)|Y|^2\ge\delta s|Y|^2 .\qedhere
		\end{align*}
	\end{proof}
	
	For the induced inner product on exterior powers,
	\begin{equation}\label{eq:wedge-identities}
		|F|^4-\tr(F^*F)^2=2\sum_{a<b}|Fe_a\wedge Fe_b|^2,\qquad
		|F|^2|Y|^2-|F^*Y|^2=\sum_a|Y\wedge Fe_a|^2,
	\end{equation}
	as follows from \(|X\wedge Y|^2=|X|^2|Y|^2-\ip XY^2\).
	
	\begin{lemma}\label{lem:curv-coercive}
		If \(\Sec_N\le-\kappa\), then for every linear \(F\colon\R^d\to T_yN\)
		and \(Y\in T_yN\),
		\begin{align}
			-\sum_{a,b}\ip{R^N(Fe_a,Fe_b)Fe_b}{Fe_a}
			&\ge\kappa\bigl(|F|^4-\tr(F^*F)^2\bigr),\label{eq:QH-alg}\\
			-\sum_a\ip{R^N(Y,Fe_a)Fe_a}{Y}
			&\ge\kappa\bigl(|F|^2|Y|^2-|F^*Y|^2\bigr).\label{eq:Q0-alg}
		\end{align}
	\end{lemma}
	
	\begin{proof}
		For all \(X,Y\), \(-\ip{R(X,Y)Y}{X}\ge\kappa(|X|^2|Y|^2-\ip XY^2)\).
		Summing over \(X=Fe_a\), \(Y=Fe_b\) and using
		\(\sum_{a,b}\ip{Fe_a}{Fe_b}^2=\tr(F^*F)^2\) gives \eqref{eq:QH-alg}.
		Summing over \(X=Y\), \(Y=Fe_a\) and using
		\(\sum_a\ip Y{Fe_a}^2=|F^*Y|^2\) gives \eqref{eq:Q0-alg}.
	\end{proof}
	
	\begin{remark}\label{rem:nondeg}
		The three conditions in Lemma~\ref{lem:linear-algebra} are equivalent up
		to the value of \(\delta\). Since \(\tr(F^*F)^2\ge\lambda_1^2\),
		\[
		s^2-\tr(F^*F)^2\ge\delta s^2\ \Longrightarrow\
		\lambda_1\le\sqrt{1-\delta}\,s\le\Bigl(1-\frac\delta2\Bigr)s ;
		\]
		choosing for \(Y\) a unit eigenvector of \(FF^*\) for \(\lambda_1\)
		shows that \eqref{eq:mixed-lower} for all \(Y\) is equivalent to
		\eqref{eq:lambda-cap}; and \eqref{eq:lambda-cap} is equivalent to a
		dilatation bound with \(\Lambda^2=(1-\delta)/\delta\). Hence, as long as
		only \(\Sec_N\le-\kappa\) is used, no condition on the singular values
		weaker than a dilatation bound makes the curvature terms coercive.
	\end{remark}
	
	\begin{theorem}\label{thm:full-schwarz}
		Let \(M\) be complete with \(\Ric^g\ge-K_gg\) and \(|D\xi|_{\op}\le L\),
		let \(\Sec_N\le-\kappa<0\), and let \(f\colon M\to N\) be
		\(\bp\)-harmonic with \(df\) of generalized dilatation at most
		\(\Lambda\). Then
		\begin{equation}\label{eq:full-schwarz}
			|df|^2\le\frac{(1+\Lambda^2)(K_g+m^2+2mL)}\kappa ,
		\end{equation}
		and \(f^*h\) is bounded by the right-hand side times \(g\).
	\end{theorem}
	
	\begin{proof}
		Since \(\nablaf_XV=S(X,\xi)+F(\mathsf BX)\),
		\begin{align*}
			\mathfrak c\sum_\mu\ip{\Jt\nablaf_{e_\mu}V}{Fe_\mu}
			&=\mathfrak c\sum_\mu\ip{\Jt S(e_\mu,\xi)}{Fe_\mu}
			+\mathfrak c\sum_\mu\ip{\Jt F(\mathsf Be_\mu)}{Fe_\mu}\\
			&\ge-\mathfrak c|S|\sqrt E-\mathfrak cLE
			\ge-|S|^2-\frac{\mathfrak c^2}4E-\mathfrak cLE,
		\end{align*}
		where we used \(\mathfrak c|S|\sqrt E=2|S|\cdot\tfrac12\mathfrak c\sqrt E
		\le|S|^2+\tfrac14\mathfrak c^2E\). Inserting this into
		\eqref{eq:full-bochner}, using \(\Ric^g\ge-K_gg\) and
		\(\mathfrak c=2m\),
		\[
		\tfrac12\Delta E\ge\calQ_H+2\calQ_0-(K_g+m^2+2mL)E .
		\]
		Lemma~\ref{lem:curv-coercive}, applied to all ordered pairs of the full
		frame, and \eqref{eq:wedge-lower} give
		\begin{align*}
			\calQ_H+2\calQ_0&=-\sum_{\mu,\nu=0}^{2m}
			\ip{R^N(Fe_\mu,Fe_\nu)Fe_\nu}{Fe_\mu}\\
			&\ge\kappa\sum_{\mu,\nu}\bigl(|Fe_\mu|^2|Fe_\nu|^2
			-\ip{Fe_\mu}{Fe_\nu}^2\bigr)\\
			&=\kappa\bigl(E^2-\tr(F^*F)^2\bigr)\ge\frac{\kappa}{1+\Lambda^2}E^2 .
		\end{align*}
		Thus \(\tfrac12\Delta E\ge AE^2-b_0E\) with \(A=\kappa/(1+\Lambda^2)\)
		and \(b_0=K_g+m^2+2mL\), and Lemma~\ref{lem:scalar-localization} with
		\(T_0=D_0=0\) gives \(E\le b_0/A\), which is \eqref{eq:full-schwarz}.
		Finally \(|df(X)|^2\le E|X|^2\).
	\end{proof}
	
	\begin{theorem}\label{thm:horizontal-schwarz}
		In addition to the hypotheses of Theorem~\ref{thm:full-schwarz} on \(M\)
		and \(N\), assume \(|\mathsf A|_{\HS}\le T\), \(|\ip Z\xi|\le z_0\) and
		\(|Z_H|\le z_1\). If \(f\) is \(\bp\)-harmonic and \(df|_H\) has
		generalized dilatation at most \(\Lambda_H\), then
		\begin{equation}\label{eq:horizontal-schwarz}
			\norm{df|_H}^2\le\max\Bigl\{\frac{2b}{\kappa\delta},
			\frac{a+\sqrt{a^2+2z_1^2}}{2\kappa\delta}\Bigr\},
		\end{equation}
		where \(\delta=(1+\Lambda_H^2)^{-1}\), \(a=K_g+3L^2+2mL+4m^2\) and
		\(b=K_g+z_0+2T^2\).
	\end{theorem}
	
	\begin{proof}
		Put \(A=\kappa\delta\) and \(F_H=F|_H\). Lemma~\ref{lem:linear-algebra}
		with \(s=U\) and \(Y=V\) gives
		\[
		U^2-\tr(F_H^*F_H)^2\ge\delta U^2,\qquad
		U|V|^2-|F_H^*V|^2\ge\delta U|V|^2 ;
		\]
		no orthogonality between \(V\) and \(df(H)\) is needed. By
		Lemma~\ref{lem:curv-coercive}, \(\calQ_H\ge AU^2\) and
		\(\calQ_0\ge AUv\). Lemma~\ref{lem:horizontal-differential} with
		\(\mathfrak c=2m\) therefore gives
		\begin{align*}
			\tfrac12\Delta U+aU&\ge AU^2+(AU-b)v-z_1\sqrt{Uv},\\
			z_1\sqrt{Uv}&=2\sqrt{\tfrac A2Uv}\cdot\frac{z_1}{\sqrt{2A}}
			\le\frac A2Uv+\frac{z_1^2}{2A}.
		\end{align*}
		Hence
		\begin{equation}\label{eq:horizontal-threshold}
			\tfrac12\Delta U\ge AU^2-aU-\frac{z_1^2}{2A}
			+\Bigl(\frac A2U-b\Bigr)v .
		\end{equation}
		On \(\{U>2b/A\}\) the last term is nonnegative, so
		\(\tfrac12\Delta U\ge AU^2-aU-D_0\) there, with \(D_0=z_1^2/(2A)\).
		Lemma~\ref{lem:scalar-localization} with \(T_0=2b/A\) and \(b_0=a\)
		gives
		\[
		U\le\max\Bigl\{\frac{2b}A,\frac{a+\sqrt{a^2+4AD_0}}{2A}\Bigr\},
		\]
		and \(4AD_0=2z_1^2\). This is \eqref{eq:horizontal-schwarz}.
	\end{proof}
	
	\begin{remark}\label{rem:local-general}
		By \eqref{eq:source-data} all constants in
		Theorem~\ref{thm:horizontal-schwarz} are controlled by \(K_g\), \(a_0\)
		and \(a_1\). If the hypotheses hold on \(B_{2R}(o)\) only, the estimates
		hold on \(B_R(o)\) after adding \(c_m(R^{-2}+\sqrt{K_g}R^{-1})\) to
		\(K_g+m^2+2mL\) and to \(a\), respectively. On a Sasakian manifold
		\(L=1\) and \(T=z_0=z_1=0\), so
		\[
		\norm{df|_H}^2\le\frac{1+\Lambda_H^2}\kappa
		\max\{2K_g,\;K_g+3+2m+4m^2\}.
		\]
		For a horizontal curve \(\gamma\),
		\(|(f\circ\gamma)'|\le\norm{df|_H}\,|\gamma'|\); integrating and taking
		the infimum over horizontal curves gives
		\(d_N(f(x),f(y))\le\sqrt{C_*}\,d_{cc}(x,y)\), where \(C_*\) is the
		right-hand side of \eqref{eq:horizontal-schwarz}.
	\end{remark}
	
	\begin{example}\label{ex:rank-one}
		The dilatation hypothesis cannot be removed. On the Heisenberg group
		\(\mathbb H^{2m+1}\) let
		\[
		\theta=dt+\sum_j(x_jdy_j-y_jdx_j),\qquad
		X_j=\partial_{x_j}+y_j\partial_t,\qquad
		Y_j=\partial_{y_j}-x_j\partial_t,\qquad JX_j=Y_j,\qquad\xi=\partial_t .
		\]
		Then \(d\theta=2\sum_jdx_j\wedge dy_j\), in accordance with
		\eqref{eq:contact-normalization}, and the structure is complete and
		Sasakian. The function \(x_1\) satisfies
		\[
		\xi x_1=0,\qquad X_1x_1=1,\qquad Y_1x_1=0,\qquad
		\Delta x_1=\sum_j(X_j^2+Y_j^2)x_1=0,\qquad|d_Hx_1|=1 .
		\]
		For a unit speed geodesic \(\gamma\) of \(\mathbb{CH}^n_\kappa\) and
		\(c>0\) put \(f_c=\gamma\circ(cx_1)\). Then
		\[
		\tau(f_c)=c^2|d_Hx_1|^2\nabla_{\gamma'}\gamma'+c\,\gamma'\,\Delta x_1=0,
		\qquad df_c(\xi)=c\,\xi(x_1)\gamma'=0,
		\]
		so \(f_c\) is \(\bp\)-harmonic. Since \(df_c=c\,\gamma'\otimes dx_1\),
		\(\norm{df_c|_H}^2=c^2\) and \(u=\bar u=c^2/4\). Thus no bound depending
		only on curvature holds for \(U\), nor for \(u\) or \(\bar u\)
		separately; the maps \(f_c\) have rank one and infinite dilatation.
	\end{example}
	
	Theorems~\ref{thm:full-schwarz} and \ref{thm:horizontal-schwarz} prove
	Theorem~\ref{thmA}.
	
	\subsection{CR maps}\label{sec:cr}
	
	\begin{theorem}\label{thm:cr-schwarz}
		Let \(M\) be complete and connected with \(\Ric^{\TW}\ge-KG\),
		\(|\mathsf A|_{\op}\le a_0\) and \(|\nabla\mathsf A|\le a_1\), let \(N\)
		satisfy \eqref{eq:curvature-contraction-assumption}, and let
		\(f\colon M\to N\) be a CR \(\bp\)-harmonic map with \(\rank_\C P\le d\),
		\(1\le d\le\min\{m,\dim_\C N\}\). Then
		\begin{equation}\label{eq:cr-parametric-schwarz}
			u\le C_t=\frac{K+ta_0^2+c_m^2/t}{\alpha_d}\qquad(t>0),
		\end{equation}
		and consequently
		\begin{equation}\label{eq:cr-schwarz-final}
			u\le C_*:=\frac{K+2c_ma_0}{\alpha_d},\qquad f^*h|_H\le C_*G .
		\end{equation}
		If the hypotheses hold on \(B_{2R}(o)\) only, then
		\begin{equation}\label{eq:cr-local-schwarz}
			\sup_{B_R(o)}u\le\frac{K+ta_0^2+c_m^2/t+c_t(R^{-2}+R^{-1})}{\alpha_d},
		\end{equation}
		where \(c_t\) depends only on \(m,K,a_0,a_1,t\).
	\end{theorem}
	
	\begin{proof}
		By Proposition~\ref{prop:exact-remainder} and \(\mathscr D_t\ge0\),
		\[
		\calL_tu\ge2\alpha_du^2-2b_tu .
		\]
		Let \(\psi=\psi_R\) be the cutoff \eqref{eq:common-cutoff}. At a positive
		maximum of \(\Phi=\psi u\),
		\begin{align*}
			0=d\Phi&=\psi\,du+u\,d\psi,\qquad du=-\frac u\psi\,d\psi,\qquad
			\Gamma_t(\psi,u)=-\frac u\psi\Gamma_t(\psi),\\
			\calL_t\Phi&=\psi\calL_tu+u\calL_t\psi+2\Gamma_t(\psi,u),\\
			0\ge\psi\calL_t\Phi&=\psi^2\calL_tu+\psi u\calL_t\psi-2u\Gamma_t(\psi)\\
			&\ge2\alpha_d\Phi^2-\bigl(2b_t+c_t(R^{-2}+R^{-1})\bigr)\Phi,
		\end{align*}
		where the last step uses Lemma~\ref{lem:anisotropic-localization}. This
		proves \eqref{eq:cr-local-schwarz}, and letting \(R\to\infty\) with \(t\)
		fixed gives \eqref{eq:cr-parametric-schwarz}. If \(a_0>0\),
		\[
		ta_0^2+\frac{c_m^2}t-2c_ma_0=\Bigl(a_0\sqrt t-\frac{c_m}{\sqrt t}\Bigr)^2,
		\]
		so \(t=c_m/a_0\) gives the first inequality in \eqref{eq:cr-schwarz-final};
		if \(a_0=0\), let \(t\to\infty\) after the spatial limit. For real
		\(X\in H\), \eqref{eq:real-complex-vector-norm} and \eqref{eq:cr-metric}
		give
		\[
		|df(X)|^2=2h\bigl(df(X^{1,0}),\overline{df(X^{1,0})}\bigr)
		\le2u\,G(X^{1,0},\overline{X^{1,0}})=u\,G(X,X)\le C_*G(X,X).\qedhere
		\]
	\end{proof}
	
	\begin{corollary}\label{cor:pinched-schwarz}
		Under the hypotheses on \(M\) of Theorem~\ref{thm:cr-schwarz}, let
		\(\dim_\C N\ge2\), \(\Bisec_N\le-\kappa\) and \(\HSC_N\ge-k_{\mathrm h}\).
		Then every CR \(\bp\)-harmonic map of complex rank at most \(d\)
		satisfies
		\begin{equation}\label{eq:pinched-schwarz}
			u\le\frac{K+2c_ma_0}{\kappa+\kappa^2/[d(k_{\mathrm h}-\kappa)]},
		\end{equation}
		and for \(N=\mathbb{CH}^n_\kappa\)
		\begin{equation}\label{eq:hyperbolic-schwarz}
			u\le\frac{d(K+2c_ma_0)}{(d+1)\kappa}.
		\end{equation}
	\end{corollary}
	
	\begin{proof}
		Insert \(\rho=\kappa^2/(k_{\mathrm h}-\kappa)\), respectively
		\(\rho=\kappa\), into \eqref{eq:effective-curvature} and
		\eqref{eq:cr-schwarz-final}.
	\end{proof}
	
	\begin{proposition}\label{prop:cr-strict}
		If \(a_0>0\) in Theorem~\ref{thm:cr-schwarz}, then \(u<C_*\) on \(M\).
	\end{proposition}
	
	\begin{proof}
		Let \(t=c_m/a_0\), so that \(C_t=C_*\), and put \(w=C_*-u\ge0\). By
		\eqref{eq:exact-remainder},
		\[
		\calL_tw-2\alpha_du\,w=-\mathscr D_t\le0 .
		\]
		If \(w\) vanishes at some point, the strong maximum principle for the
		elliptic operator \(\calL_t-2\alpha_du\) gives \(w\equiv0\); then
		\(\calL_tu=0\) and \(2\alpha_du(u-C_*)=0\), so \(\mathscr D_t\equiv0\).
		In particular \(v_c=0\), hence \(r=0\) and
		\(f_{i0}=f_{0i}=\nablaf_{Z_i}r=0\), and
		\[
		0=2t\sum_i\Bigl|f_{i0}-\frac{ic_m}tf_i\Bigr|^2=\frac{2c_m^2}tu
		=\frac{2c_m^2}tC_*>0,
		\]
		a contradiction.
	\end{proof}
	
	\subsection{The Reeb differential and transverse holomorphicity}
	
	\begin{theorem}\label{thm:reeb-bound}
		Under the hypotheses of Theorem~\ref{thm:cr-schwarz},
		\begin{equation}\label{eq:reeb-schwarz-final}
			|r|^2\le\frac{a_0^2}\kappa,\qquad |df(\xi)|^2\le\frac{2a_0^2}\kappa .
		\end{equation}
		In particular \(f\) is foliated if \(M\) is Sasakian.
	\end{theorem}
	
	\begin{proof}
		Since \(|\xi|=1\) and \(g\) is complete, the flow \(\phi_s\) of \(\xi\) is
		complete. By \eqref{eq:cr-first-derivatives},
		\[
		\xi v_c=2\Rea\ip{\nablaf_\xi r}{r}=2\Rea\ip{f_{00}}{r}=0,
		\]
		so \(v_c\circ\phi_s=v_c\), and the open set
		\(\Omega=\{\kappa v_c>a_0^2\}\) satisfies \(\phi_s(\Omega)=\Omega\).
		Suppose \(\Omega\ne\varnothing\), let \(x\in\Omega\) and put
		\(y(s)=u(\phi_s(x))\). By Theorem~\ref{thm:cr-schwarz} and
		\eqref{eq:cr-reeb-lower},
		\begin{align*}
			0&\le y(s)\le C_*,\\
			y''(s)&\ge2s_0(\phi_s(x))+2\bigl(\kappa v_c(x)-a_0^2\bigr)y(s)
			\ge2\bigl(\kappa v_c(x)-a_0^2\bigr)y(s)\ge0 .
		\end{align*}
		A bounded convex function on \(\R\) is constant; hence \(y''=0\), and
		\(\kappa v_c(x)-a_0^2>0\) forces \(y\equiv0\). Thus \(u=0\), that is
		\(f_i\equiv0\), on \(\Omega\). Since \(\nabla\theta=0\) and \(\nabla J=0\),
		\[
		\theta(\nabla_{Z_{\bar j}}Z_i)=0,\qquad
		J\nabla_{Z_{\bar j}}Z_i=\nabla_{Z_{\bar j}}(JZ_i)
		=i\nabla_{Z_{\bar j}}Z_i,
		\]
		so \(\nabla_{Z_{\bar j}}Z_i\in H^{1,0}\), and \(df^{1,0}|_{H^{1,0}}=0\)
		on \(\Omega\). Therefore, on \(\Omega\),
		\[
		2i\delta_{ij}r=f_{i\bar j}
		=\nablaf_{Z_{\bar j}}\bigl(df^{1,0}(Z_i)\bigr)
		-df^{1,0}(\nabla_{Z_{\bar j}}Z_i)=0 .
		\]
		Hence \(v_c=0\) on \(\Omega\), which contradicts
		\(\kappa v_c>a_0^2\ge0\) on \(\Omega\). Thus \(\Omega=\varnothing\), and
		\(|df(\xi)|^2=2|r|^2\le2a_0^2/\kappa\).
	\end{proof}
	
	\begin{proposition}\label{prop:transverse-hol}
		Let \(M\) be Sasakian and \(f\colon M\to N\) a foliated CR map. Locally
		the Reeb foliation has a K\"ahler quotient \(\pi\colon\mathcal O\to Q\),
		and \(f=\widehat f\circ\pi\) for a holomorphic map
		\(\widehat f\colon Q\to N\). The quotient metric satisfies
		\(\pi^*g_Q|_H=G\), and its Ricci tensor pulls back to \(\Ric^{\TW}\).
	\end{proposition}
	
	\begin{proof}
		Cartan's formula and \(\mathsf A=0\) give
		\begin{align*}
			\mathcal L_\xi\theta&=d(\theta(\xi))+\iota_\xi d\theta=0,\\
			(\mathcal L_\xi g)(X,Y)&=g(D_X\xi,Y)+g(X,D_Y\xi)
			=g(JX,Y)+g(X,JY)=0,\\
			(\mathcal L_\xi J)X&=[\xi,JX]-J[\xi,X]=\nabla_\xi(JX)-J\nabla_\xi X=0 .
		\end{align*}
		Hence a flow box carries tensors \(g_Q,J_Q,\omega_Q\) with
		\[
		\pi^*g_Q|_H=G,\qquad d\pi\circ J=J_Q\circ d\pi,\qquad
		\pi^*\omega_Q=\tfrac12d\theta,\qquad d\omega_Q=0,
		\]
		and integrability of the CR structure gives integrability of \(J_Q\).
		Since \(df(\xi)=0\), the map descends, and
		\[
		d\widehat f\circ J_Q\circ d\pi=df\circ J=\Jt\circ df
		=\Jt\circ d\widehat f\circ d\pi,
		\]
		so \(\widehat f\) is holomorphic. For projectable horizontal fields
		\(d\pi(\nabla_XY)=\nabla^Q_{d\pi X}d\pi Y\), because the induced
		connection preserves \(g_Q\) and has vanishing projected torsion.
		Consequently \(d\pi(R^\nabla(X,Y)Z)=R^Q(d\pi X,d\pi Y)d\pi Z\), and taking
		the horizontal trace proves the Ricci identity.
	\end{proof}
	
	\begin{corollary}\label{cor:cr-liouville}
		If \(M\) is complete Sasakian with \(\Ric^{\TW}\ge0\) and
		\(\Bisec_N\le-\kappa<0\), then every CR \(\bp\)-harmonic map
		\(M\to N\) is constant.
	\end{corollary}
	
	\begin{proof}
		Theorems~\ref{thm:cr-schwarz} and \ref{thm:reeb-bound} with \(K=a_0=0\)
		give \(u=0\) and \(df(\xi)=0\); by \eqref{eq:real-complex-energy},
		\(df=0\).
	\end{proof}
	
	\begin{corollary}\label{cor:cc-cr}
		Under the hypotheses of Theorem~\ref{thm:cr-schwarz}, put
		\(C_V=2a_0^2/\kappa\). Then for every \(\eps>0\)
		\begin{equation}\label{eq:metric-distance}
			f^*h\le(1+\eps)C_*G+(1+\eps^{-1})C_V\,\theta^2,\qquad
			d_N(f(x),f(y))\le\sqrt{C_*}\,d_{cc}(x,y).
		\end{equation}
	\end{corollary}
	
	\begin{proof}
		For \(X\in H\) and \(s\in\R\),
		\begin{align*}
			|df(X+s\xi)|^2&\le\bigl(\sqrt{C_*}|X|+\sqrt{C_V}|s|\bigr)^2\\
			&\le(1+\eps)C_*|X|^2+(1+\eps^{-1})C_Vs^2 .
		\end{align*}
		For a horizontal curve \(\gamma\) from \(x\) to \(y\),
		\(d_N(f(x),f(y))\le\int|df(\dot\gamma)|\le\sqrt{C_*}\int|\dot\gamma|\);
		taking the infimum over such curves proves the second inequality.
	\end{proof}
	
	Theorems~\ref{thm:cr-schwarz} and \ref{thm:reeb-bound},
	Corollary~\ref{cor:pinched-schwarz} and
	Propositions~\ref{prop:cr-strict} and \ref{prop:transverse-hol} prove
	Theorem~\ref{thmB}.
	
	\section{Equality, stability and compactness}\label{sec:equality}
	
	Throughout this section \(M\) is a connected complete Sasakian manifold
	with \(\Ric^{\TW}\ge-KG\), \(K>0\), the target satisfies
	\eqref{eq:curvature-contraction-assumption}, and \(f\) is a CR
	\(\bp\)-harmonic map of complex rank at most \(d\). By
	Theorem~\ref{thm:reeb-bound}, \(f\) is foliated, so
	\eqref{eq:sasaki-deficit-identity} holds, and
	Theorem~\ref{thm:cr-schwarz} with \(a_0=0\) gives \(0\le u\le C_0\).
	
	\subsection{Equality and models}
	
	\begin{theorem}\label{thm:hyperbolic-equality}
		Let \(f\colon M\to\mathbb{CH}^n_\kappa\), \(1\le d\le\min\{m,n\}\), and
		suppose
		\begin{equation}\label{eq:hyperbolic-equality-assumption}
			u(x_0)=C_d=\frac{dK}{(d+1)\kappa}
		\end{equation}
		for some \(x_0\in M\). Then the transverse differential of \(f\) is
		parallel and of complex rank \(d\) everywhere. Every local K\"ahler
		quotient \(Q\) of the Reeb foliation admits a local product decomposition
		\begin{equation}\label{eq:hyperbolic-product}
			Q=S\times Q',\qquad\dim_\C S=d,\qquad\HSC_S=-\frac{2K}{d+1},
		\end{equation}
		the induced map \(\widehat f\) is constant along \(Q'\), and on \(S\) it
		is a totally geodesic holomorphic homothety,
		\begin{equation}\label{eq:hyperbolic-homothety}
			\widehat f^{\,*}h|_{TS}=\lambda g_S,\qquad\lambda=\frac K{(d+1)\kappa}.
		\end{equation}
	\end{theorem}
	
	\begin{proof}
		For \(\mathbb{CH}^n_\kappa\) we take \(\rho=\kappa\); then
		\(C_0=C_d\), and \eqref{eq:complex-hyperbolic-contraction} gives
		\(D_N=\mathscr C-\kappa u^2-\kappa|P|^2=0\). Thus \(w=C_d-u\ge0\)
		satisfies
		\begin{equation}\label{eq:hyperbolic-equality-equation}
			\begin{aligned}
				\Bigl(\Delta-2\kappa\Bigl(1+\frac1d\Bigr)u\Bigr)w
				&=-2\sum_{i,j}|f_{ij}|^2-2D_R-2\kappa\Bigl(|P|^2-\frac{u^2}d\Bigr)\le0 .
			\end{aligned}
		\end{equation}
		Since \(M\) is connected and \(w(x_0)=0\), the strong maximum principle
		gives
		\[
		w\equiv0\ \text{on }M\quad\Longrightarrow\quad
		\mathscr D_{\mathrm S}\equiv0\ \text{on }M,
		\]
		and every nonnegative term on the right of
		\eqref{eq:hyperbolic-equality-equation} vanishes:
		\begin{equation}\label{eq:hyperbolic-zero-terms}
			f_{ij}=f_{i\bar j}=f_{i0}=0,\qquad D_R=0,\qquad D_P=0 .
		\end{equation}
		List the eigenvalues of \(P\) as \(\lambda_1,\dots,\lambda_d\), appending
		zeros if \(\rank_\C P<d\). At every point,
		\begin{align*}
			\sum_{i=1}^d\lambda_i&=u=C_d,\qquad\sum_{i=1}^d\lambda_i^2=|P|^2,\\
			0=D_P&=\sum_{i=1}^d\Bigl(\lambda_i-\frac{C_d}d\Bigr)^2,\\
			\lambda_i&=\frac{C_d}d=\lambda>0\qquad(1\le i\le d).
		\end{align*}
		An appended zero would contradict the last line; hence
		\(\rank_\C P=d\) and \(P^2=\lambda P\). Moreover \eqref{eq:hyperbolic-zero-terms}
		gives \(\nabla(df|_H)=0\) and \(\nabla P=0\). The endomorphism
		\(\Pi=\lambda^{-1}P\) satisfies
		\begin{equation}\label{eq:parallel-projection}
			\Pi^2=\Pi,\qquad\Pi^*=\Pi,\qquad\nabla\Pi=0,\qquad[\Pi,J]=0 .
		\end{equation}
		Its image \(\mathcal V\) and kernel \(\mathcal N\) are orthogonal,
		parallel, \(J\)-invariant subbundles of \(H\). For horizontal \(X\),
		\([\xi,X]=\nabla_\xi X\) because \(\mathsf A=0\), and
		\(\nabla_\xi\Pi=0\); hence \(\mathcal V\) and \(\mathcal N\) are Reeb
		invariant and descend to orthogonal parallel complex subbundles on the
		local quotient of Proposition~\ref{prop:transverse-hol}. The local de~Rham
		decomposition theorem yields
		\[
		(Q,g_Q,J_Q)=(S,g_S,J_S)\times(Q',g_{Q'},J_{Q'}),\qquad
		TS=\mathcal V,\qquad TQ'=\mathcal N .
		\]
		
		Let \(\nabla^{Q,N}\) be the connection on \(T^*Q\otimes\widehat f^*TN\)
		induced by the Levi-Civita connections of \(Q\) and \(N\). Then
		\begin{equation}\label{eq:parallel-map-quotient}
			\nabla^{Q,N}d\widehat f=0,\qquad d\widehat f|_{TQ'}=0,\qquad
			\widehat f^{\,*}h|_{TS}=\lambda g_S .
		\end{equation}
		The Ricci identity for \(d\widehat f\) gives
		\begin{align}
			0&=\bigl(\nabla^{Q,N}_X\nabla^{Q,N}_Y-\nabla^{Q,N}_Y\nabla^{Q,N}_X
			-\nabla^{Q,N}_{[X,Y]}\bigr)d\widehat f\,(Z)\notag\\
			&=R^N(d\widehat fX,d\widehat fY)d\widehat fZ
			-d\widehat f\bigl(R^Q(X,Y)Z\bigr).\label{eq:curvature-intertwining}
		\end{align}
		For \(X,Y,Z\in T^{1,0}S\), inserting
		\eqref{eq:complex-hyperbolic-tensor} and \eqref{eq:parallel-map-quotient},
		\begin{align*}
			d\widehat f\bigl(R^S(X,\bar Y)Z\bigr)
			&=R^N(d\widehat fX,\overline{d\widehat fY})d\widehat fZ\\
			&=-\kappa\,h(d\widehat fX,\overline{d\widehat fY})\,d\widehat fZ
			-\kappa\,h(d\widehat fZ,\overline{d\widehat fY})\,d\widehat fX\\
			&=-\kappa\lambda\,d\widehat f\bigl(g_S(X,\bar Y)Z+g_S(Z,\bar Y)X\bigr).
		\end{align*}
		Since \(d\widehat f|_{TS}\) is injective,
		\begin{equation}\label{eq:factor-curvature-tensor}
			R^S(X,\bar Y)Z=-\kappa\lambda\bigl(g_S(X,\bar Y)Z+g_S(Z,\bar Y)X\bigr),
		\end{equation}
		and therefore
		\[
		\HSC_S=-2\kappa\lambda=-\frac{2K}{d+1},\qquad
		\Ric_S=-(d+1)\kappa\lambda\,g_S=-Kg_S .
		\]
		A holomorphic map with \(\nabla^{Q,N}d\widehat f=0\) is totally
		geodesic, which completes the proof.
	\end{proof}
	
	\begin{corollary}\label{cor:global-product}
		In Theorem~\ref{thm:hyperbolic-equality}, assume that the Reeb foliation
		is given by a Riemannian submersion \(\pi\colon(M,g)\to(Q,g_Q)\) with
		connected fibres, \(\pi^*g_Q|_H=G\), onto a complete simply connected
		K\"ahler manifold \(Q\). Then \(Q=S\times Q'\) globally, where \(S\) is
		the simply connected complex space form of holomorphic sectional
		curvature \(-2K/(d+1)\), and \(\widehat f\) is the projection onto \(S\)
		followed by a totally geodesic holomorphic homothety onto a complex
		hyperbolic subspace of \(\mathbb{CH}^n_\kappa\).
	\end{corollary}
	
	\begin{proof}
		The parallel projection \eqref{eq:parallel-projection} descends to a
		parallel projection on \(Q\), and the de~Rham decomposition theorem for
		complete simply connected manifolds gives \(Q=S\times Q'\) with complete
		simply connected factors. Equation~\eqref{eq:factor-curvature-tensor}
		identifies \(S\). The differential vanishes along \(Q'\), so
		\(\widehat f\) factors through \(S\). For every geodesic \(\gamma\) of
		\(S\),
		\[
		\nabla^N_{\dot\gamma}\,d\widehat f(\dot\gamma)
		=d\widehat f\bigl(\nabla^S_{\dot\gamma}\dot\gamma\bigr)=0,
		\]
		so the image is the complete totally geodesic complex hyperbolic subspace
		determined by \(d\widehat f(TS)\) at one point.
	\end{proof}
	
	\begin{corollary}\label{cor:cr-equality}
		Let \(\Bisec_N\le-\kappa<0\), and suppose that a CR \(\bp\)-harmonic map
		satisfies \(u(x_0)=K/\kappa\). Then \(\dim_\C N=1\), the transverse
		differential is parallel of complex rank one, every local K\"ahler
		quotient splits locally as \(\Sigma\times Q'\) with a complex curve
		\(\Sigma\) of curvature \(-K\), the induced map factors through
		\(\Sigma\), and
		\begin{equation}\label{eq:curve-homothety}
			\widehat f^{\,*}h|_{T\Sigma}=\frac K\kappa\,g_\Sigma .
		\end{equation}
		The curvature of \(N\) equals \(-\kappa\) along the image.
	\end{corollary}
	
	\begin{proof}
		Take \(\rho=0\) and \(d=m\) in \eqref{eq:sasaki-deficit-identity}, so
		that \(C_0=K/\kappa\). The strong maximum principle applied to
		\(K/\kappa-u\) gives \(u\equiv K/\kappa\) and
		\[
		f_{ij}=0,\qquad D_R=0,\qquad
		\sum_{i,j}\bigl(-R(f_i,\bar f_i,f_j,\bar f_j)
		-\kappa|f_i|^2|f_j|^2\bigr)=0 .
		\]
		Each summand is nonnegative by \eqref{eq:bisectional}, hence
		\begin{equation}\label{eq:coarse-equality-curvature}
			R(f_i,\bar f_i,f_j,\bar f_j)=-\kappa|f_i|^2|f_j|^2\qquad(i,j).
		\end{equation}
		Some \(f_i\ne0\), and \eqref{eq:coarse-equality-curvature} with \(i=j\)
		gives \(\HSC_N(f_i)=-\kappa\); Corollary~\ref{cor:no-hol-curv} forces
		\(\dim_\C N=1\), so \(\rank_\C P=1\) and
		\[
		P^2=\frac K\kappa P,\qquad\Pi=\frac\kappa KP,\qquad\Pi^2=\Pi,\qquad
		\nabla\Pi=0 .
		\]
		As in Theorem~\ref{thm:hyperbolic-equality}, the local de~Rham theorem
		gives \(Q=\Sigma\times Q'\) with \(d\widehat f|_{TQ'}=0\) and
		\(\widehat f^{\,*}h=\lambda g_\Sigma\), \(\lambda=K/\kappa\). For a real
		\(g_\Sigma\)-orthonormal pair \(X,J_\Sigma X\),
		\begin{align*}
			|d\widehat fX|_h^2&=|d\widehat f(J_\Sigma X)|_h^2=\lambda,\qquad
			h\bigl(d\widehat fX,d\widehat f(J_\Sigma X)\bigr)=0,\\
			h\bigl(R^N(d\widehat fX,d\widehat fJ_\Sigma X)d\widehat fJ_\Sigma X,
			d\widehat fX\bigr)&=K_N\lambda^2,\\
			h\bigl(d\widehat f(R^\Sigma(X,J_\Sigma X)J_\Sigma X),d\widehat fX\bigr)
			&=\lambda K_\Sigma .
		\end{align*}
		By \eqref{eq:curvature-intertwining} the last two expressions coincide,
		and \eqref{eq:coarse-equality-curvature} gives \(K_N=-\kappa\). Hence
		\begin{equation}\label{eq:curve-gauss-curvature}
			K_\Sigma=\lambda K_N=-K,
		\end{equation}
		which proves the corollary.
	\end{proof}
	
	\begin{remark}\label{rem:curve-normalization}
		For a complex curve of constant curvature \(-b<0\) one has
		\(\Bisec_N=\HSC_N=-b\), so Corollary~\ref{cor:cr-equality} applies with
		\((\kappa,\rho,d)=(b,0,1)\), \(\alpha_1=b\), and gives the constant
		\(K/b\). For \(N=\mathbb{CH}^1_\kappa\) the normalization of
		\eqref{eq:complex-hyperbolic-tensor} gives \(b=2\kappa\), hence
		\(K/b=K/(2\kappa)=C_1\), the constant of
		Theorem~\ref{thm:hyperbolic-equality} for \(d=1\). Both descriptions give
		the same remainder:
		\begin{align*}
			|P|^2&=u^2,\qquad D_P=0,\qquad\mathscr C=bu^2=2\kappa u^2,\\
			\mathscr C-bu^2&=0=\mathscr C-\kappa u^2-\kappa|P|^2,\qquad
			\mathscr D_{\mathrm S}=2\sum_{i,j}|f_{ij}|^2+2D_R .
		\end{align*}
	\end{remark}
	
	\begin{example}\label{ex:sharp-cr}
		Fix \(1\le d\le\min\{m,n\}\) and put \(\lambda=K/((d+1)\kappa)\). Let
		\(\iota\colon\mathbb{CH}^d_\kappa\to\mathbb{CH}^n_\kappa\) be a totally
		geodesic holomorphic inclusion, and equip its domain \(S\) with
		\(g_S=\lambda^{-1}\iota^*h\). Then
		\[
		\iota^*h=\lambda g_S,\qquad\HSC_S=-2\kappa\lambda=-\frac{2K}{d+1},
		\qquad\Ric_S=-(d+1)\kappa\,\iota^*h=-Kg_S .
		\]
		Put \(\mathcal B=S\times\C^{m-d}\) with the product metric
		\(g_{\mathcal B}=g_S+g_{\mathrm{eucl}}\), and choose \(\alpha\) with
		\(d\alpha=\omega_{\mathcal B}\), which is possible since \(\mathcal B\) is
		contractible. On \(M=\mathcal B\times\R\) define
		\begin{equation}\label{eq:model-contact}
			\theta=dt+2\alpha,\qquad g=\pi^*g_{\mathcal B}+\theta^2,\qquad
			\xi=\partial_t .
		\end{equation}
		Horizontal lifts of vector fields \(X,Y\) on \(\mathcal B\) satisfy
		\begin{equation}\label{eq:model-bracket}
			[X^H,Y^H]=[X,Y]^H-2\omega_{\mathcal B}(X,Y)\xi,\qquad[\xi,X^H]=0 .
		\end{equation}
		The connection defined by
		\[
		\nabla_{X^H}Y^H=(\nabla^{\mathcal B}_XY)^H,\qquad\nabla\xi=0,\qquad
		\nabla_\xi X^H=0
		\]
		preserves the lifted metric and complex structure, and
		\[
		T^\nabla(X^H,Y^H)=2\omega_{\mathcal B}(X,Y)\xi,\qquad
		T^\nabla(\xi,X^H)=0 .
		\]
		It is therefore the Tanaka--Webster connection, so that
		\begin{equation}\label{eq:model-curvature}
			\mathsf A=0,\qquad R^\nabla(X^H,Y^H)Z^H=(R^{\mathcal B}(X,Y)Z)^H,\qquad
			\Ric^{\TW}=-Kg_S\oplus0 .
		\end{equation}
		For \(X=X_S+X_0\in TS\oplus T\C^{m-d}\),
		\[
		\Ric^{\TW}(X,X)+KG(X,X)=-K|X_S|^2+K\bigl(|X_S|^2+|X_0|^2\bigr)
		=K|X_0|^2\ge0 .
		\]
		The metric is complete. For a curve \(\gamma=(c,t)\),
		\[
		|\dot c|_{g_{\mathcal B}}\le|\dot\gamma|_g,\qquad
		|\dot t|\le|\theta(\dot\gamma)|+2|\alpha|_{g_{\mathcal B}}
		|\dot c|_{g_{\mathcal B}} .
		\]
		If \(\gamma\) has finite length, \(c\) has finite length and stays in a
		compact set on which \(|\alpha|\le C_\gamma\), so
		\(\int|\dot t|\le(1+2C_\gamma)L_g(\gamma)<\infty\) and \(t\) converges.
		The map \(f=\iota\circ\mathrm{pr}_S\circ\pi\) satisfies
		\[
		(\nabla df)|_{H\times H}=0,\qquad df(\xi)=0,\qquad
		f_{i\bar j}=f_{00}=0,
		\]
		the horizontal derivative being taken with the Tanaka--Webster
		connection; hence \(f\) is CR and \(\bp\)-harmonic, and
		\[
		P=\lambda\,\Id_{T^{1,0}S}\oplus0,\qquad u=d\lambda=C_d .
		\]
		Thus equality holds in \eqref{eq:hyperbolic-schwarz} for every
		admissible \(d\).
	\end{example}
	
	\subsection{Stability, compactness and gaps}
	
	\begin{theorem}\label{thm:sasaki-quantitative}
		Assume \(\rho>0\). For \(0\le\phi\in C_c^\infty(M)\) put
		\(I_\phi=\int_M\phi\,\mathscr D_{\mathrm S}\dd\). Then
		\begin{equation}\label{eq:Iphi}
			I_\phi=2\alpha_d\int_M\phi\,uw\dd-\int_Mw\,\Delta\phi\dd .
		\end{equation}
		If \(0\le\phi\le1\), \(\phi=1\) on a measurable set \(\Omega\), and
		\(b_\phi=\|(-\Delta\phi)_+\|_\infty\), then
		\begin{equation}\label{eq:local-sasaki-remainder}
			\int_\Omega\mathscr D_{\mathrm S}\dd\le(2K+b_\phi)
			\int_{\operatorname{supp}\phi}(C_0-u)\dd .
		\end{equation}
		With the Hilbert--Schmidt norm on \(\operatorname{End}H^{1,0}\) and the
		metric trace on its covariant derivative,
		\begin{align}
			\int_M\phi\,|\nabla P|^2\dd&\le C_0\,I_\phi,\label{eq:P-gradient}\\
			\int_M\phi\,\Bigl|P^2-\frac udP\Bigr|^2\dd
			&\le\frac{C_0^2}{2\rho}\,I_\phi,\label{eq:P-polynomial}\\
			\int_M\phi\,\Bigl|P^2-\frac{C_0}dP\Bigr|^2\dd
			&\le\frac{C_0^2}\rho\,I_\phi+\frac{2C_0^2}{d^2}\int_M\phi\,w^2\dd .
			\label{eq:P-fixed-polynomial}
		\end{align}
	\end{theorem}
	
	\begin{proof}
		By \eqref{eq:sasaki-deficit-identity} and \eqref{eq:laplacian-symmetry},
		\begin{align*}
			\int_M\phi\,\mathscr D_{\mathrm S}\dd
			&=2\alpha_d\int_M\phi\,uw\dd-\int_M\phi\,\Delta w\dd\\
			&=2\alpha_d\int_M\phi\,uw\dd-\int_Mw\,\Delta\phi\dd .
		\end{align*}
		For \(0\le\phi\le1\) with \(\phi|_\Omega=1\), using
		\(\mathbf1_\Omega\le\phi\), \(0\le u\le C_0\) and \(2\alpha_dC_0=2K\),
		\begin{align*}
			\int_\Omega\mathscr D_{\mathrm S}\dd&\le\int_M\phi\,\mathscr D_{\mathrm S}\dd
			\le2K\int_M\phi\,w\dd+\int_Mw(-\Delta\phi)_+\dd\\
			&\le(2K+b_\phi)\int_{\operatorname{supp}\phi}w\dd .
		\end{align*}
		Since \(f_{i\bar k}=f_{i0}=0\), metric compatibility gives
		\[
		P_{i\bar j,k}=\ip{f_{ik}}{f_j},\qquad
		P_{i\bar j,\bar k}=\ip{f_i}{f_{jk}},\qquad P_{i\bar j,0}=0,
		\]
		and therefore
		\begin{align*}
			|\nabla P|^2&=\sum_{i,j,k}\bigl(|\ip{f_{ik}}{f_j}|^2
			+|\ip{f_i}{f_{jk}}|^2\bigr)\\
			&\le2u\sum_{i,k}|f_{ik}|^2\le C_0\,\mathscr D_{\mathrm S},
		\end{align*}
		which gives \eqref{eq:P-gradient}. With the eigenvalues
		\(\lambda_1,\dots,\lambda_d\in[0,u]\) (zeros included),
		\begin{align*}
			\Bigl|P^2-\frac udP\Bigr|^2
			&=\sum_{i=1}^d\lambda_i^2\Bigl(\lambda_i-\frac ud\Bigr)^2
			\le u^2\sum_{i=1}^d\Bigl(\lambda_i-\frac ud\Bigr)^2\\
			&=u^2D_P\le\frac{C_0^2}{2\rho}\,\mathscr D_{\mathrm S},
		\end{align*}
		which gives \eqref{eq:P-polynomial}. Finally, \(u=C_0-w\) and
		\(|P|\le u\le C_0\) give
		\begin{align*}
			\Bigl|P^2-\frac{C_0}dP\Bigr|^2
			&=\Bigl|P^2-\frac udP-\frac wdP\Bigr|^2
			\le2\Bigl|P^2-\frac udP\Bigr|^2+\frac{2w^2}{d^2}|P|^2\\
			&\le\frac{C_0^2}\rho\,\mathscr D_{\mathrm S}+\frac{2C_0^2}{d^2}w^2,
		\end{align*}
		and integration gives \eqref{eq:P-fixed-polynomial}.
	\end{proof}
	
	\begin{proposition}\label{prop:pointwise-stability}
		Assume \(\rho>0\), let \(x_0\in M\) and \(R>0\). There are constants
		\(h_R\) and \(A_R\), depending only on \(g|_{B_{4R}(x_0)}\), \(R\), \(K\),
		\(\kappa\), \(\rho\) and \(d\), such that
		\begin{gather}
			\int_{B_{2R}(x_0)}(C_0-u)\dd\le h_R\bigl(C_0-u(x_0)\bigr),
			\label{eq:integral-deficit-stability}\\
			\int_{B_R(x_0)}\Bigl(|\nabla P|^2+\Bigl|P^2-\frac udP\Bigr|^2\Bigr)\dd
			\le A_R\bigl(C_0-u(x_0)\bigr).\label{eq:pointwise-stability}
		\end{gather}
		One may take
		\begin{equation}\label{eq:stability-constant}
			A_R=\Bigl(C_0+\frac{C_0^2}{2\rho}\Bigr)(2K+b_\phi)\,h_R ,
		\end{equation}
		where \(\phi\in C_c^\infty(B_{2R}(x_0))\), \(0\le\phi\le1\), \(\phi=1\)
		on \(B_R(x_0)\), and \(b_\phi=\|(-\Delta\phi)_+\|_\infty\).
	\end{proposition}
	
	\begin{proof}
		By \eqref{eq:sasaki-deficit-identity}, \(w=C_0-u\ge0\) satisfies
		\[
		-\Delta w+a\,w=\mathscr D_{\mathrm S}\ge0,\qquad a=2\alpha_du,\qquad
		0\le a\le2\alpha_dC_0=2K .
		\]
		In local coordinates
		\(\Delta=(\det g)^{-1/2}\partial_i\bigl((\det g)^{1/2}g^{ij}\partial_j\bigr)\)
		is a uniformly elliptic operator in divergence form on relatively
		compact coordinate balls, and \(w\) is a nonnegative supersolution of
		\(Lw=\Delta w-aw\le0\), with zeroth-order coefficient \(-a\le0\). Let
		\(\ell=2m+1\ge3\). The weak Harnack inequality \cite[Theorem~8.18]{GT}
		gives, for \(0<p<\ell/(\ell-2)\) and balls \(B_{4r}(y)\) contained in a
		normal coordinate chart,
		\begin{equation}\label{eq:weak-harnack-exponents}
			\Bigl(\frac1{\Vol B_{2r}(y)}\int_{B_{2r}(y)}w^p\dd\Bigr)^{1/p}
			\le h_{y,r,p}\inf_{B_r(y)}w,
		\end{equation}
		where \(h_{y,r,p}\) depends only on the metric coefficients, \(r\), \(p\)
		and the bound \(2K\). Taking \(p=1\),
		\begin{equation}\label{eq:weak-harnack-local}
			\int_{B_r(y)}w\dd\le\Vol B_{2r}(y)\,h_{y,r,1}\inf_{B_r(y)}w
			=:h_{y,r}\Vol B_r(y)\inf_{B_r(y)}w .
		\end{equation}
		Choose normal coordinate balls \(U_j=B_{r_j}(y_j)\), \(j=1,\dots,N\), with
		\[
		y_j\in\overline{B_{2R}(x_0)},\qquad
		0<r_j<\min\Bigl\{\frac R2,\frac{\operatorname{inj}_g(y_j)}4\Bigr\},
		\qquad x_0\in U_1,
		\]
		covering the compact set \(\overline{B_{2R}(x_0)}\). Then
		\(d_g(x_0,y_j)+4r_j<2R+2R=4R\), so
		\(B_{4r_j}(y_j)\Subset B_{4R}(x_0)\). Since
		\(\overline{B_{2R}(x_0)}\) is connected, the intersection graph of the
		cover is connected; choose a rooted spanning tree with root \(1\) and
		order the vertices so that the parent \(\pi(j)\) satisfies \(\pi(j)<j\).
		Put \(h_j=h_{y_j,r_j}\) and
		\begin{align*}
			v_j&=\Vol U_j,\qquad v_{j,\pi(j)}=\Vol\bigl(U_j\cap U_{\pi(j)}\bigr)>0,\\
			A_1&=h_1v_1,\qquad A_j=\frac{h_jv_j}{v_{j,\pi(j)}}A_{\pi(j)}
			\quad(j\ge2),\qquad h_R=\sum_{j=1}^NA_j .
		\end{align*}
		By \eqref{eq:weak-harnack-local},
		\begin{align*}
			\int_{U_1}w\dd&\le h_1v_1\inf_{U_1}w\le A_1\,w(x_0),\\
			\inf_{U_j}w&\le\frac1{v_{j,\pi(j)}}\int_{U_j\cap U_{\pi(j)}}w\dd,\\
			\int_{U_j}w\dd&\le h_jv_j\inf_{U_j}w
			\le\frac{h_jv_j}{v_{j,\pi(j)}}\int_{U_{\pi(j)}}w\dd
			\le\frac{h_jv_j}{v_{j,\pi(j)}}A_{\pi(j)}w(x_0)=A_jw(x_0),
		\end{align*}
		by induction along the tree. Summation gives
		\begin{equation}\label{eq:deficit-harnack}
			\int_{B_{2R}(x_0)}w\dd\le\sum_{j=1}^N\int_{U_j}w\dd
			\le\Bigl(\sum_{j=1}^NA_j\Bigr)w(x_0)=h_R\,w(x_0),
		\end{equation}
		which is \eqref{eq:integral-deficit-stability}. With the cutoff \(\phi\)
		of the statement, \eqref{eq:local-sasaki-remainder} gives
		\[
		\int_{B_R(x_0)}\mathscr D_{\mathrm S}\dd
		\le(2K+b_\phi)\int_{B_{2R}(x_0)}w\dd\le(2K+b_\phi)\,h_R\,w(x_0),
		\]
		and the pointwise inequalities in the proof of
		Theorem~\ref{thm:sasaki-quantitative} give
		\begin{align*}
			|\nabla P|^2+\Bigl|P^2-\frac udP\Bigr|^2
			&\le\Bigl(C_0+\frac{C_0^2}{2\rho}\Bigr)\mathscr D_{\mathrm S},\\
			\int_{B_R(x_0)}\Bigl(|\nabla P|^2+\Bigl|P^2-\frac udP\Bigr|^2\Bigr)\dd
			&\le\Bigl(C_0+\frac{C_0^2}{2\rho}\Bigr)(2K+b_\phi)\,h_R\,w(x_0).\qedhere
		\end{align*}
	\end{proof}
	
	\begin{theorem}\label{thm:compactness}
		Let \(1\le d\le\min\{m,n\}\), \(x_0\in M\), and let
		\(f_\nu\colon M\to\mathbb{CH}^n_\kappa\) be CR \(\bp\)-harmonic maps of
		complex rank at most \(d\) with
		\begin{equation}\label{eq:near-extremal-sequence}
			u_\nu(x_0)\longrightarrow C_d=\frac{dK}{(d+1)\kappa}.
		\end{equation}
		Then there are holomorphic isometries \(a_\nu\) of \(\mathbb{CH}^n_\kappa\)
		and a subsequence such that \(a_\nu\circ f_\nu\) converges in
		\(C^\infty_{\mathrm{loc}}\) to a map \(f_\infty\) with
		\begin{equation}\label{eq:limit-metric}
			u_\infty\equiv C_d,\qquad P_\infty=\frac K{(d+1)\kappa}\,\Pi_\infty,
			\qquad\nabla\Pi_\infty=0 .
		\end{equation}
	\end{theorem}
	
	\begin{proof}
		The holomorphic isometry group of \(\mathbb{CH}^n_\kappa\) acts
		transitively. Fix \(y_0\), choose \(a_\nu\) with
		\(a_\nu(f_\nu(x_0))=y_0\), and keep the notation \(f_\nu\) for the
		normalized maps. By Theorems~\ref{thm:cr-schwarz} and
		\ref{thm:reeb-bound} and Corollary~\ref{cor:pinched-schwarz} with
		\(a_0=0\),
		\[
		df_\nu(\xi)=0,\qquad\tau(f_\nu)=0,\qquad
		|df_\nu|_{\op}^2\le u_\nu\le C_d,\qquad|df_\nu|^2=2u_\nu ,
		\]
		hence \(d_N(f_\nu(x),y_0)\le\sqrt{C_d}\,d_g(x,x_0)\). Since
		\(\exp_{y_0}\) is a global diffeomorphism, put
		\(z_\nu=\exp_{y_0}^{-1}\circ f_\nu\colon M\to T_{y_0}N\cong\R^{2n}\). For
		a relatively compact source coordinate domain \(U'\) put
		\(R_{U'}=\sqrt{C_d}\sup_{x\in U'}d_g(x,x_0)\). Then
		\begin{align*}
			|z_\nu(x)|&=d_N(f_\nu(x),y_0)\le R_{U'}\quad(x\in U'),\qquad
			\|z_\nu\|_{L^p(U')}\le R_{U'}\Vol(U')^{1/p},\\
			c_{U'}^{-1}I&\le(h_{\alpha\beta}(z))\le c_{U'}I,\qquad
			\sup_{|z|\le R_{U'}}|\partial^kh(z)|\le c_{U',k},\\
			\sup_{U'}|\partial z_\nu|&\le c_{U'}\sup_{U'}|df_\nu|_{\op}
			\le c_{U'}\sqrt{C_d}.
		\end{align*}
		With the Christoffel symbols \(\Gamma^c_{ab}\) of \(g\) and
		\(\Gamma^\alpha_{\beta\gamma}\) of \(h\), the harmonic map equation reads
		\begin{equation}\label{eq:coordinate-harmonic}
			\begin{aligned}
				g^{ab}\partial_a\partial_bz_\nu^\alpha&=F_\nu^\alpha,\\
				F_\nu^\alpha&=g^{ab}\Gamma^c_{ab}\partial_cz_\nu^\alpha
				-g^{ab}\Gamma^\alpha_{\beta\gamma}(z_\nu)\partial_az_\nu^\beta
				\partial_bz_\nu^\gamma .
			\end{aligned}
		\end{equation}
		For \(U\Subset U'\) and \(1<p<\infty\), the interior \(L^p\) estimate
		\cite[Theorem~9.11]{GT} gives
		\begin{align*}
			\|F_\nu\|_{L^\infty(U')}
			&\le c_{U'}\bigl(\|\partial z_\nu\|_{L^\infty(U')}
			+\|\partial z_\nu\|_{L^\infty(U')}^2\bigr)\le c'_{U'},\\
			\|z_\nu\|_{W^{2,p}(U)}&\le c_{U,U',p}\bigl(\|z_\nu\|_{L^p(U')}
			+\|F_\nu\|_{L^p(U')}\bigr)\le c'_{U,U',p}.
		\end{align*}
		Choose \(p>\ell=2m+1\), \(0<\gamma<1-\ell/p\) and \(U_0\Subset U\).
		Sobolev embedding and the interior Schauder estimate
		\cite[Theorem~6.2]{GT} yield
		\begin{align*}
			\|z_\nu\|_{C^{1,\gamma}(U)}&\le c\,\|z_\nu\|_{W^{2,p}(U)}\le c',\\
			\|F_\nu\|_{C^{0,\gamma}(U)}&\le c\bigl(1+\|z_\nu\|_{C^{1,\gamma}(U)}\bigr)^3
			\le c',\\
			\|z_\nu\|_{C^{2,\gamma}(U_0)}&\le c_{U_0,U}\bigl(\|z_\nu\|_{C^0(U)}
			+\|F_\nu\|_{C^{0,\gamma}(U)}\bigr)\le c'_{U_0,U}.
		\end{align*}
		Differentiating \eqref{eq:coordinate-harmonic} on nested domains gives
		\(\sup_\nu\|z_\nu\|_{C^{k,\gamma}(F)}\le c_{F,k}\) for every compact
		\(F\subset M\) and \(k\ge1\). A diagonal subsequence satisfies
		\[
		z_{\nu_j}\to z_\infty,\qquad f_{\nu_j}\to f_\infty=\exp_{y_0}\circ z_\infty
		\qquad\text{in }C^\infty_{\mathrm{loc}},
		\]
		and the conditions \(df(\xi)=0\), \(f_{\bar i}=0\), \(\tau(f)=0\) and
		\(\rank_\C P\le d\) pass to the limit. Moreover
		\[
		u_\infty(x)=\lim_ju_{\nu_j}(x)\le C_d,\qquad
		u_\infty(x_0)=\lim_ju_{\nu_j}(x_0)=C_d .
		\]
		Theorem~\ref{thm:hyperbolic-equality} applied to \(f_\infty\) gives
		\(u_\infty\equiv C_d\), \(\rank_\C P_\infty=d\),
		\(P_\infty=\lambda\Pi_\infty\) with \(\lambda=C_d/d=K/((d+1)\kappa)\),
		and \(\nabla\Pi_\infty=0\).
	\end{proof}
	
	\begin{theorem}\label{thm:gap}
		Let \(1\le d\le\min\{m,n\}\), and let \(F\subset M\) be a nonempty compact
		set such that at no point of \(F\) the local K\"ahler quotient has a
		local product factor of complex dimension \(d\) and holomorphic sectional
		curvature \(-2K/(d+1)\). Then there is \(\eps_F>0\), depending on the
		pseudo-Hermitian structure of \(M\), on \(F\) and on \(K,\kappa,d,n\),
		such that every CR \(\bp\)-harmonic map \(f\colon M\to\mathbb{CH}^n_\kappa\)
		of complex rank at most \(d\) satisfies
		\begin{equation}\label{eq:compact-gap}
			\sup_Fu\le C_d-\eps_F .
		\end{equation}
	\end{theorem}
	
	\begin{proof}
		If \eqref{eq:compact-gap} fails, there are \(x_\nu\in F\) and maps
		\(f_\nu\) with \(u_\nu(x_\nu)\to C_d\); after passing to a subsequence,
		\(x_\nu\to x\in F\). Normalize by holomorphic isometries so that
		\(f_\nu(x_\nu)=y_0\). Then
		\[
		d_N(f_\nu(x),y_0)=d_N\bigl(f_\nu(x),f_\nu(x_\nu)\bigr)
		\le\sqrt{C_d}\,d_g(x,x_\nu)\to0,
		\]
		and the estimates in the proof of Theorem~\ref{thm:compactness} give a
		subsequence \(f_{\nu_j}\to f_\infty\) in \(C^\infty_{\mathrm{loc}}\). In
		particular \(\|u_{\nu_j}-u_\infty\|_{C^0(F)}\to0\), and
		\begin{align*}
			|u_\infty(x)-C_d|
			&\le|u_\infty(x)-u_\infty(x_{\nu_j})|
			+\|u_\infty-u_{\nu_j}\|_{C^0(F)}+|u_{\nu_j}(x_{\nu_j})-C_d|
			\longrightarrow0 .
		\end{align*}
		Hence \(u_\infty(x)=C_d\), and Theorem~\ref{thm:hyperbolic-equality}
		gives the excluded product factor in the local K\"ahler quotient of
		\(M\) at \(x\), a contradiction.
	\end{proof}
	
	Theorem~\ref{thm:hyperbolic-equality},
	Proposition~\ref{prop:pointwise-stability} and
	Theorem~\ref{thm:compactness} prove Theorem~\ref{thmC}.
	
	\section{Foliation and rigidity}\label{sec:rigidity}
	
	In this section \(M^{2m+1}\), \(m\ge2\), is a connected complete Sasakian
	manifold, \(N\) is a strongly seminegative K\"ahler manifold, and
	\(f\colon M\to N\) is \(\bp\)-harmonic. 
	
	\subsection{A Caccioppoli inequality and foliation}
	
	\begin{theorem}\label{thm:caccioppoli}
		For every compactly supported Lipschitz function \(\chi\),
		\begin{equation}\label{eq:siu-caccioppoli}
			\int_M\chi^2\calI\dd\le4\bigl(c_m^{\mathrm S}\bigr)^2
			\int_M|d\chi|^2E\dd .
		\end{equation}
	\end{theorem}
	
	\begin{proof}
		For smooth \(\chi\), the divergence theorem,
		\eqref{eq:positive-divergence} and \eqref{eq:Y-control} give
		\begin{align*}
			\int_M\chi^2\calI\dd&=-2\int_M\chi\ip{d\chi}{\mathcal Y}\dd
			\le2\int_M|\chi|\,|d\chi|\,c_m^{\mathrm S}\sqrt E\sqrt\calI\dd\\
			&\le\frac12\int_M\chi^2\calI\dd
			+2\bigl(c_m^{\mathrm S}\bigr)^2\int_M|d\chi|^2E\dd .
		\end{align*}
		Absorbing the first term gives \eqref{eq:siu-caccioppoli}. A Lipschitz
		function with compact support is approximated in \(W^{1,2}\) by smooth
		functions supported in a fixed compact set, on which \(E\) and \(\calI\)
		are bounded.
	\end{proof}
	
	Inequality \eqref{eq:siu-caccioppoli} involves neither second
	derivatives of \(\chi\) nor any lower curvature bound on \(M\). It
	therefore applies whenever the energy admits cutoff functions of small
	weighted Dirichlet energy. We say that \(f\) satisfies the
	\emph{energy cutoff condition} if there are
	\(\chi_k\in\operatorname{Lip}_c(M)\) with
	\begin{equation}\label{eq:cutoff-condition}
		0\le\chi_k\le1,\qquad\chi_k\to1\ \text{locally uniformly},\qquad
		\int_ME|d\chi_k|^2\dd\to0 .
	\end{equation}
	Under \eqref{eq:cutoff-condition}, inequality \eqref{eq:siu-caccioppoli}
	gives \(\calI\equiv0\), and hence \(\nablaf V=0\). A parallel section
	need not vanish, however, so an additional identity is required to
	conclude \(V=0\). It is provided by the following contact current, which
	pairs the Reeb differential with the horizontal differential rotated
	by \(J\).
	
	\begin{lemma}\label{lem:contact-current}
		The horizontal one-form \(\alpha(X)=h(V,df(JX))\), \(\alpha(\xi)=0\),
		satisfies \(|\alpha|\le|V|\sqrt U\) and
		\begin{equation}\label{eq:contact-current}
			\divg\alpha^\sharp=-2m|V|^2+\sum_{a=1}^{2m}\ip{\nablaf_{e_a}V}{df(Je_a)},
		\end{equation}
		where \(\nablaf_{e_a}V=\beta(\xi,e_a)\) is computed with the
		Tanaka--Webster connection. Moreover
		\begin{equation}\label{eq:horizontal-reeb-derivative-norm}
			\sum_{a=1}^{2m}|\nablaf_{e_a}V|^2=2s .
		\end{equation}
	\end{lemma}
	
	\begin{proof}
		Since the trace of \(\mathsf C\) vanishes and \(\alpha^\sharp\) is
		horizontal, \(\divg\alpha^\sharp=\sum_a(\nabla_{e_a}\alpha)(e_a)\), and
		\(\nabla J=0\) gives
		\[
		\divg\alpha^\sharp=\sum_a\ip{\nablaf_{e_a}V}{df(Je_a)}
		+\sum_a\ip{V}{\beta(Je_a,e_a)} .
		\]
		For a pair \(X_\ell,JX_\ell\) of the frame, the identity
		\(\beta(X,Y)-\beta(Y,X)=df(T^\nabla(X,Y))\) and \eqref{eq:tw-torsion} give
		\[
		\beta(JX_\ell,X_\ell)-\beta(X_\ell,JX_\ell)
		=-df\bigl(T^\nabla(X_\ell,JX_\ell)\bigr)=-2g(JX_\ell,JX_\ell)V=-2V,
		\]
		and summing over the \(m\) pairs gives \eqref{eq:contact-current}. The
		bound for \(\alpha\) is \(|\alpha|^2=\sum_a\ip{V}{df(Je_a)}^2\le|V|^2U\).
		Next,
		\begin{align*}
			\nablaf_{Z_i}V&=\nablaf_{Z_i}r+\overline{\nablaf_{Z_{\bar i}}r}
			=f_{0i}+\overline{f_{0\bar i}},\qquad
			f_{0i}\in f^*T^{1,0}N,\quad\overline{f_{0\bar i}}\in f^*T^{0,1}N,\\
			|\nablaf_{Z_i}V|^2&=|f_{0i}|^2+|f_{0\bar i}|^2,\qquad
			\nablaf_{Z_i}V=\frac{\nablaf_{e_i}V-i\nablaf_{Je_i}V}{\sqrt2},\\
			2|\nablaf_{Z_i}V|^2&=|\nablaf_{e_i}V|^2+|\nablaf_{Je_i}V|^2,
		\end{align*}
		the last identity because \(\nablaf_{e_i}V\) and \(\nablaf_{Je_i}V\) are
		real. Summing over \(i\) gives \eqref{eq:horizontal-reeb-derivative-norm}.
	\end{proof}
	
	\begin{theorem}\label{thm:capacity-rigidity}
		If \(f\) satisfies \eqref{eq:cutoff-condition}, then
		\[
		df(\xi)=0,\qquad f_{i\bar j}=f_{\bar ji}=0,\qquad
		\calR(\calU_{ij})=0\quad\text{for all }i,j;
		\]
		thus \(f\) is foliated, \(\bp\)-pluriharmonic and harmonic.
	\end{theorem}
	
	\begin{proof}
		For every compact \(F\subset M\), local uniform convergence of \(\chi_k\)
		and \eqref{eq:siu-caccioppoli} give
		\[
		0\le\int_F\calI\dd\le\liminf_k\int_M\chi_k^2\calI\dd
		\le4\bigl(c_m^{\mathrm S}\bigr)^2\lim_k\int_ME|d\chi_k|^2\dd=0 .
		\]
		Hence \(\calI\equiv0\), so \(s=\omega_0=0\), \(\mathsf T^\circ=0\) and
		\(\calS=0\); in particular \(f_{0i}=f_{0\bar i}=f_{00}=0\), that is
		\(\nablaf V=0\). Then \eqref{eq:contact-current} gives
		\(\divg\alpha^\sharp=-2m|V|^2\), and for the cutoffs
		\begin{align*}
			2m\int\chi_k^2v\dd&=2\int\chi_k\ip{d\chi_k}{\alpha^\sharp}\dd
			\le2\int|\chi_k|\sqrt v\,|d\chi_k|\sqrt U\dd\\
			&\le m\int\chi_k^2v\dd+\frac1m\int|d\chi_k|^2U\dd ,
		\end{align*}
		so \(\int\chi_k^2v\le m^{-2}\int|d\chi_k|^2E\to0\), and \(V=0\). Then
		\(r=\mathsf Y=0\), \(\tr\mathsf T=0\) and \(\mathsf T=\mathsf T^\circ=0\),
		and \eqref{eq:basic-commutators} gives \(f_{\bar ji}=f_{i\bar j}=0\). The
		nonnegative terms \(\calR(\calU_{ij})\) sum to \(2\calS=0\), hence vanish.
		Finally \(\tau(f)=2m\Jt V=0\).
	\end{proof}
	
	Without the cutoff condition, the same two identities still bound the
	Reeb energy on the support of a cutoff function in terms of its
	weighted Dirichlet energy. This quantitative form is recorded in the
	next proposition.
	
	\begin{proposition}\label{prop:approx}
		For a compactly supported Lipschitz function \(\chi\) put
		\(\eps_\chi=\int_ME|d\chi|^2\dd\) and \(H_\chi=\int_M\chi^2E\dd\). Then
		\begin{equation}\label{eq:approx-V}
			\int_M\chi^2|df(\xi)|^2\dd\le\frac4{3m^2}\,\eps_\chi
			+\frac{4c_m^{\mathrm S}}{3m}\sqrt{\eps_\chi H_\chi}.
		\end{equation}
	\end{proposition}
	
	\begin{proof}
		Multiplying \eqref{eq:contact-current} by \(\chi^2\) and integrating,
		\[
		2m\int\chi^2v\dd=2\int\chi\ip{d\chi}{\alpha^\sharp}\dd
		+\int\chi^2\sum_a\ip{\nablaf_{e_a}V}{df(Je_a)}\dd .
		\]
		Since \(|\alpha|\le\sqrt{vE}\),
		\[
		2\Bigl|\int\chi\ip{d\chi}{\alpha^\sharp}\dd\Bigr|
		\le\frac m2\int\chi^2v\dd+\frac2m\eps_\chi .
		\]
		By Cauchy--Schwarz, \eqref{eq:horizontal-reeb-derivative-norm},
		\(2s\le\calI\) and \eqref{eq:siu-caccioppoli},
		\begin{align*}
			\Bigl|\int\chi^2\sum_a\ip{\nablaf_{e_a}V}{df(Je_a)}\dd\Bigr|
			&\le\Bigl(\int\chi^2\,2s\dd\Bigr)^{1/2}\Bigl(\int\chi^2U\dd\Bigr)^{1/2}\\
			&\le\Bigl(\int\chi^2\calI\dd\Bigr)^{1/2}H_\chi^{1/2}
			\le2c_m^{\mathrm S}\sqrt{\eps_\chi H_\chi}.
		\end{align*}
		Absorbing \(\tfrac m2\int\chi^2v\) gives
		\(\tfrac{3m}2\int\chi^2v\le\tfrac2m\eps_\chi
		+2c_m^{\mathrm S}\sqrt{\eps_\chi H_\chi}\), which is \eqref{eq:approx-V}.
	\end{proof}
	
	\subsection{Energy growth and volume growth}
	
	The cutoff condition \eqref{eq:cutoff-condition} is a weighted capacity
	condition at infinity. It is implied by an integral condition on the
	energy \(\calE_f(R)=\int_{B_R(o)}E\dd\) of geodesic balls, and, when the
	energy density is bounded by a Schwarz estimate, by a volume growth
	condition.
	
	\begin{proposition}\label{prop:capacity-growth}
		Condition \eqref{eq:cutoff-condition} holds if \(E\equiv0\) or if
		\begin{equation}\label{eq:energy-growth-integral}
			\int_{R_0}^\infty\frac{R\,dR}{\calE_f(R)}=\infty,\qquad
			\calE_f(R)>0\quad(R\ge R_0).
		\end{equation}
		In particular it holds if \(\int_ME\dd<\infty\) or if
		\(\calE_f(R)\le CR^2\log R\log\log R\) for large \(R\).
	\end{proposition}
	
	\begin{proof}
		Assume \(E\not\equiv0\) and choose \(N_0\) with \(\calE_f(2^{N_0})>0\).
		For \(N_0\le N\le M_1\) put
		\[
		c_j=4^{-j}\calE_f(2^{j+1}),\qquad L_{N,M_1}=\sum_{j=N}^{M_1}c_j^{-1},
		\qquad d_j=\frac{c_j^{-1}}{L_{N,M_1}} .
		\]
		Let \(\chi_{N,M_1}\) be the radial function equal to \(1\) on
		\(B_{2^N}\), decreasing linearly in \(r\) by \(d_j\) on
		\(B_{2^{j+1}}\setminus B_{2^j}\), and vanishing outside
		\(B_{2^{M_1+1}}\). Since \(\sum_jd_j=1\) and \(|d\chi_{N,M_1}|\le
		2^{-j}d_j\) on the \(j\)-th annulus,
		\begin{align*}
			\int_ME|d\chi_{N,M_1}|^2\dd
			&\le\sum_{j=N}^{M_1}4^{-j}d_j^2\,\calE_f(2^{j+1})
			=\sum_{j=N}^{M_1}c_jd_j^2=L_{N,M_1}^{-1}.
		\end{align*}
		Monotonicity of \(\calE_f\) gives
		\[
		\frac{3\cdot4^j}{2\calE_f(2^{j+1})}
		\le\int_{2^j}^{2^{j+1}}\frac{R\,dR}{\calE_f(R)}
		\le\frac{3\cdot4^j}{2\calE_f(2^j)},
		\]
		so the series with terms \(4^j/\calE_f(2^j)\) and
		\(4^j/\calE_f(2^{j+1})\) differ only by an index shift and the factor
		\(4\); thus \eqref{eq:energy-growth-integral} implies
		\(\sum_jc_j^{-1}=\infty\). Choose \(M_1(N)\) with
		\(L_{N,M_1(N)}\ge N\). The functions \(\chi_{N,M_1(N)}\) satisfy
		\(\chi_{N,M_1(N)}\to1\) locally uniformly and
		\(\int_ME|d\chi_{N,M_1(N)}|^2\le N^{-1}\). If \(\calE_f\le C\), the
		integrand in \eqref{eq:energy-growth-integral} is at least \(R/C\).
		Under the stated growth,
		\[
		\frac R{\calE_f(R)}\ge\frac1{CR\log R\log\log R},\qquad
		\int^\infty\frac{dR}{R\log R\log\log R}=\infty .\qedhere
		\]
	\end{proof}
	
	For maps of bounded generalized dilatation, Theorem~\ref{thm:full-schwarz}
	bounds the energy density by a constant, so the energy growth in
	Proposition~\ref{prop:capacity-growth} is controlled by the volume
	growth of \(M\). The full estimate is needed here: the horizontal
	estimate of Theorem~\ref{thm:horizontal-schwarz} controls \(U\) only,
	not the Reeb energy \(v\).
	
	\begin{theorem}\label{thm:volume-rigidity}
		Let \(\Ric^g\ge-K_gg\), \(\Sec_N\le-\kappa<0\), and let \(f\) have
		generalized dilatation at most \(\Lambda\). If
		\begin{equation}\label{eq:volume-growth}
			\Vol B_R(o)\le CR^2\log R\log\log R
		\end{equation}
		for large \(R\), then \(f\) is foliated, \(\bp\)-pluriharmonic and
		harmonic.
	\end{theorem}
	
	\begin{proof}
		Since \(D\xi=J\), we may take \(L=1\) in Theorem~\ref{thm:full-schwarz},
		which gives
		\[
		E\le C_E:=\frac{(1+\Lambda^2)(K_g+m^2+2m)}\kappa,\qquad
		\calE_f(R)\le C_E\Vol B_R(o)\le CC_ER^2\log R\log\log R .
		\]
		Proposition~\ref{prop:capacity-growth} gives \eqref{eq:cutoff-condition},
		and Theorem~\ref{thm:capacity-rigidity} completes the proof.
	\end{proof}
	
	Theorem~\ref{thm:volume-rigidity} is Theorem~\ref{thmD}.
	
	\subsection{Rank rigidity}
	
	Theorem~\ref{thm:capacity-rigidity} also yields the vanishing of the
	Siu curvature terms \(\calR(\calU_{ij})\). If the target is strongly
	negative, this vanishing is an algebraic relation between the
	holomorphic and antiholomorphic parts \(p\) and \(q\) of the horizontal
	differential. Where the real rank is at least three, the relation
	separates the two types; unique continuation then propagates the
	separation to all of \(M\). Combined with the Liouville theorem for CR
	maps, Corollary~\ref{cor:cr-liouville}, this bounds the rank.
	
	\begin{theorem}\label{thm:strong-negative-structure}
		In addition to \eqref{eq:cutoff-condition}, let \(N\) be strongly
		negative. If \(\rank_\R df\ge3\) at some point, then \(f\) is foliated
		and either CR or anti-CR. If moreover \(\Ric^{\TW}\ge0\) and
		\(\Bisec_N\le-\kappa_1<0\), then \(\rank_\R df\le2\) everywhere.
	\end{theorem}
	
	\begin{proof}
		By Theorem~\ref{thm:capacity-rigidity} and strong negativity,
		\begin{equation}\label{eq:rank-tensor-relation}
			q_j\otimes\bar p_i=q_i\otimes\bar p_j\qquad\text{for all }i,j .
		\end{equation}
		At a point where \(p\ne0\) and \(q\ne0\), choose \(i\) with \(p_i\ne0\);
		then \(q_i\ne0\), since otherwise \(q_j\otimes\bar p_i=0\) forces
		\(q=0\). Equality of the decomposable tensors gives \(c_j\in\C\) with
		\[
		q_j=c_jq_i,\qquad p_j=\bar c_jp_i .
		\]
		For a horizontal vector represented by \(z\in\C^m\), the
		\((1,0)\)-part of its image is
		\[
		\sum_j(z_jp_j+\bar z_jq_j)=\zeta p_i+\bar\zeta q_i,\qquad
		\zeta=\sum_j\bar c_jz_j ,
		\]
		and since \(df(\xi)=0\) this gives \(\rank_\R df\le2\). Let \(\Omega\) be
		a connected component of the open set \(\{\rank_\R df\ge3\}\) and put
		\[
		\Omega_p=\{x\in\Omega:p(x)\ne0\},\qquad\Omega_q=\{x\in\Omega:q(x)\ne0\}.
		\]
		These sets are open; they are disjoint by the computation above; and
		they cover \(\Omega\), because \(p(x)=q(x)=0\) and \(df(\xi)=0\) give
		\(df_x=0\). Connectedness forces \(q\equiv0\) or \(p\equiv0\) on
		\(\Omega\).
		
		Regard \(q\) as a section of \(\calE=(H^{0,1})^*\otimes f^*T^{1,0}N\).
		Pluriharmonicity, foliation and \(\mathsf A=0\) give
		\(\nabla^{\calE}_{Z_i}q=0\) and \(\nabla^{\calE}_\xi q=0\). The Webster
		connection Laplacian is
		\begin{equation}\label{eq:full-bundle-laplacian}
			\Delta^{\calE}q=\sum_{\mu=0}^{2m}\bigl(\nabla^{\calE}_{e_\mu}
			\nabla^{\calE}_{e_\mu}q-\nabla^{\calE}_{D_{e_\mu}e_\mu}q\bigr).
		\end{equation}
		Define \((\nabla^{\calE})^2_{X,Y}=\nabla^{\calE}_X\nabla^{\calE}_Y
		-\nabla^{\calE}_{\nabla_XY}\). Since the trace of \(D-\nabla\) vanishes
		and \(D_\xi\xi=0\),
		\[
		\Delta^{\calE}q=\sum_i\bigl((\nabla^{\calE})^2_{Z_i,Z_{\bar i}}
		+(\nabla^{\calE})^2_{Z_{\bar i},Z_i}\bigr)q+(\nabla^{\calE})^2_{\xi,\xi}q .
		\]
		The alternation of the Hessian is
		\[
		(\nabla^{\calE})^2_{X,Y}q-(\nabla^{\calE})^2_{Y,X}q
		=R^{\calE}(X,Y)q-\nabla^{\calE}_{T^\nabla(X,Y)}q,\qquad
		T^\nabla(Z_i,Z_{\bar i})=2i\xi .
		\]
		Since \(\nabla\) preserves \(H^{1,0}\) and \(\xi\),
		\begin{align*}
			(\nabla^{\calE})^2_{Z_{\bar i},Z_i}q
			&=\nabla^{\calE}_{Z_{\bar i}}\bigl(\nabla^{\calE}_{Z_i}q\bigr)
			-\nabla^{\calE}_{\nabla_{Z_{\bar i}}Z_i}q=0,\qquad
			(\nabla^{\calE})^2_{\xi,\xi}q=0,\\
			(\nabla^{\calE})^2_{Z_i,Z_{\bar i}}q
			&=(\nabla^{\calE})^2_{Z_{\bar i},Z_i}q+R^{\calE}(Z_i,Z_{\bar i})q
			-2i\nabla^{\calE}_\xi q=R^{\calE}(Z_i,Z_{\bar i})q,
		\end{align*}
		and therefore \(\Delta^{\calE}q=\sum_iR^{\calE}(Z_i,Z_{\bar i})q\). Let
		\(\Gamma^c_{ab}\) be the Christoffel symbols of \(g\), and write
		\(\nabla^{\calE}_a=\partial_a+\Omega_a\) in a local trivialization. Then
		\begin{equation}\label{eq:bundle-coordinate-expansion}
			\begin{aligned}
				0={}&g^{ab}\partial_a\partial_bq
				+\bigl(2g^{ab}\Omega_a-g^{ac}\Gamma^b_{ac}\bigr)\partial_bq\\
				&+\Bigl(g^{ab}\bigl(\partial_a\Omega_b+\Omega_a\Omega_b
				-\Gamma^c_{ab}\Omega_c\bigr)-\sum_iR^{\calE}(Z_i,Z_{\bar i})\Bigr)q .
			\end{aligned}
		\end{equation}
		The principal part is the scalar operator \(g^{ab}\partial_a\partial_b\),
		which is elliptic because
		\[
		\sum_{a=1}^{2m}\zeta(e_a)^2+\zeta(\xi)^2=|\zeta|_g^2\qquad(\zeta\in T^*M).
		\]
		Writing \(\mathbf q=(\Rea q^A,\Ima q^A)_A\) in a local frame, on every
		relatively compact coordinate domain \(U\),
		\[
		\bigl|g^{ab}\partial_a\partial_b\mathbf q\bigr|^2
		\le c_U\Bigl(\sum_a|\partial_a\mathbf q|^2+|\mathbf q|^2\Bigr).
		\]
		Aronszajn's unique continuation theorem \cite{Aronszajn} therefore gives
		\(q\equiv0\) on \(M\) if \(q\) vanishes on \(\Omega\); exchanging barred
		and unbarred indices, the same holds for \(p\). Hence \(f\) is foliated
		and CR or anti-CR.
		
		Under the additional assumptions, \(f\) is CR with respect to \(\Jt\) or
		to \(-\Jt\); the latter structure has the same bisectional curvature, and
		\(\tau(f)=0=2m(-\Jt)V\). Corollary~\ref{cor:cr-liouville} shows that \(f\)
		is constant, contradicting \(\rank_\R df\ge3\). Hence
		\(\rank_\R df\le2\) everywhere.
	\end{proof}
	
	The hypothesis \(\Ric^{\TW}\ge0\) in the second part of
	Theorem~\ref{thm:strong-negative-structure} cannot be removed: the
	following example satisfies all other hypotheses of the theorem, but
	the real rank of its differential is \(2m\ge4\).
	
	\begin{example}\label{ex:finite-energy}
		Let \(m\ge2\), let \(\Gamma\) be a torsion-free cocompact lattice in the
		holomorphic isometry group of \(\mathbb{CH}^m_{\kappa_0}\), and let
		\(X=\mathbb{CH}^m_{\kappa_0}/\Gamma\) be the compact quotient, with
		K\"ahler metric \(g_X\) and K\"ahler form \(\omega_X\). By
		\eqref{eq:complex-hyperbolic-tensor},
		\[
		-2\kappa_0\le\Bisec_X\le-\kappa_0,\qquad
		\Ric_X=-(m+1)\kappa_0\,g_X .
		\]
		Hence \(g_X\) is K\"ahler--Einstein and \([\omega_X]\) is a positive
		multiple of \(2\pi c_1(K_X)\). Choose \(c>0\) such that
		\([c\,\omega_X/\pi]\) lies in the image of \(H^2(X;\mathbb Z)\), and let
		\(\pi\colon M\to X\) be a principal circle bundle with connection form
		\(\eta\) satisfying \(d\eta=2c\,\pi^*\omega_X\) (Boothby--Wang). Put
		\[
		\theta=\eta,\qquad G=\pi^*(c\,g_X),\qquad g=G+\theta^2 .
		\]
		By the computation in Example~\ref{ex:sharp-cr}, \((M,\theta)\) is a
		compact Sasakian manifold of dimension \(2m+1\) with
		\[
		\mathsf A=0,\qquad
		\Ric^{\TW}=\pi^*\Ric_{c g_X}=-(m+1)\kappa_0\,\pi^*g_X
		=-\frac{(m+1)\kappa_0}{c}\,G ,
		\]
		so \(\Ric^{\TW}\) is negative definite. The projection
		\(f=\pi\colon M\to X\) is CR and foliated. Its transverse differential is
		an isometry from \((H,G)\) onto \((TX,c\,g_X)\), so
		\((\nabla df)|_{H\times H}=0\) and \(f_{i\bar j}=f_{00}=0\). Hence \(f\) is
		\(\bp\)-harmonic and
		\[
		df(\xi)=0,\qquad \rank_\R df=2m\ge4,\qquad
		\int_M E\dd=2mc\,\Vol_g M<\infty .
		\]
		Since \(M\) is compact, \(\chi_k\equiv1\) satisfies
		\eqref{eq:cutoff-condition}. Compact quotients of complex hyperbolic
		space are strongly negative \cite{Siu}, and
		\(\Bisec_X\le-\kappa_0<0\). Thus all hypotheses of
		Theorem~\ref{thm:strong-negative-structure} hold except
		\(\Ric^{\TW}\ge0\). In accordance with the first part of the theorem,
		\(f\) is foliated and CR; the conclusion \(\rank_\R df\le2\) of the
		second part fails.
	\end{example}
	
	\appendix
	\section{Connections and commutators}\label{app:connections}
	
	\begin{proof}[Proof of Lemma~\ref{lem:connection-difference}]
		Put \(\widetilde D=\nabla+\mathsf C\); it is a connection because
		\(\mathsf C\) is tensorial. It is metric if and only if
		\(g(\mathsf C(X,Y),W)+g(Y,\mathsf C(X,W))=0\). Pairing the four terms of
		\(\mathsf C\) with \(W\) and symmetrizing in \(Y,W\) gives
		\begin{align*}
			&-g(JX,Y)\theta(W)-g(JX,W)\theta(Y)-g(\mathsf AX,Y)\theta(W)
			-g(\mathsf AX,W)\theta(Y)\\
			&+\theta(Y)g(\mathsf AX,W)+\theta(W)g(\mathsf AX,Y)
			+\theta(X)\{g(JY,W)+g(Y,JW)\}\\
			&+\theta(Y)g(JX,W)+\theta(W)g(JX,Y)=0 .
		\end{align*}
		For the torsion,
		\(\widetilde D_XY-\widetilde D_YX-[X,Y]=T^\nabla(X,Y)+\mathsf C(X,Y)
		-\mathsf C(Y,X)\), and the symmetry of \(\mathsf A\) gives
		\[
		\mathsf C(X,Y)-\mathsf C(Y,X)=-2g(JX,Y)\xi+\theta(Y)\mathsf AX
		-\theta(X)\mathsf AY=-T^\nabla(X,Y).
		\]
		By uniqueness of the Levi-Civita connection, \(\widetilde D=D\). Putting
		\(Y=\xi\) and using \(\nabla\xi=0\), \(J\xi=\mathsf A\xi=0\) gives
		\(D_X\xi=JX+\mathsf AX\). Moreover \(\divg\xi=\tr(J+\mathsf A)=0\) and
		\(\sum_\mu\mathsf C(e_\mu,e_\mu)=-\tr(\mathsf A)\xi=0\).
	\end{proof}
	
	\begin{proof}[Proof of Lemma~\ref{lem:source-reeb}]
		For a vector field \(Y\) with \(\divg Y=0\), commuting covariant
		derivatives in a normal frame gives
		\begin{align*}
			\sum_iD_i(D_iY_j+D_jY_i)&=(\Delta^DY)_j+D_j(\divg Y)+\Ric_{jk}Y_k\\
			&=(\Delta^DY+\Ric^\sharp Y)_j ,
		\end{align*}
		so \(\Delta^DY+\Ric^\sharp Y=2\divg(\operatorname{sym}DY)\). For \(Y=\xi\),
		\(\operatorname{sym}D\xi=\mathsf A\) by \eqref{eq:reeb-geometry}, which is
		\eqref{eq:z-div}. With \(\mathsf C=D-\nabla\),
		\begin{align*}
			R^g(X,Y)W={}&R^\nabla(X,Y)W+(\nabla_X\mathsf C)(Y,W)
			-(\nabla_Y\mathsf C)(X,W)\\
			&+\mathsf C(X,\mathsf C(Y,W))-\mathsf C(Y,\mathsf C(X,W))
			+\mathsf C(T^\nabla(X,Y),W).
		\end{align*}
		Take \(Y=W=\xi\) and \(X\) horizontal. Then
		\begin{align*}
			R^\nabla(X,\xi)\xi&=0,\qquad\mathsf C(\xi,\xi)=0,\qquad
			\mathsf C(X,\xi)=JX+\mathsf AX,\qquad\mathsf C(\xi,X)=JX,\\
			(\nabla_X\mathsf C)(\xi,\xi)&=0,\qquad
			(\nabla_\xi\mathsf C)(X,\xi)=(\nabla_\xi\mathsf A)X,\qquad
			T^\nabla(X,\xi)=-\mathsf AX ,
		\end{align*}
		and therefore
		\[
		R^g(X,\xi)\xi=-(\nabla_\xi\mathsf A)X-\mathsf C(\xi,JX+\mathsf AX)
		-\mathsf C(\mathsf AX,\xi).
		\]
		Since \(X\) is horizontal,
		\[
		\mathsf C(\xi,JX+\mathsf AX)=-X+J\mathsf AX,\qquad
		\mathsf C(\mathsf AX,\xi)=J\mathsf AX+\mathsf A^2X,
		\]
		and substitution gives \eqref{eq:reeb-curvature}. The operator norm
		bounds imply
		\[
		|\ip{J\mathsf AX}{X}|\le a_0|X|^2,\qquad|\mathsf AX|^2\le a_0^2|X|^2,
		\qquad|\ip{(\nabla_\xi\mathsf A)X}{X}|\le a_1|X|^2,
		\]
		and substitution into \eqref{eq:reeb-curvature} gives the lower bound.
	\end{proof}
	
	\begin{proof}[Proof of \eqref{eq:basic-commutators} and
		\eqref{eq:ricci-identity}]
		Since \(\nablaf_Xdf(Y)-\nablaf_Ydf(X)=df([X,Y])\),
		\begin{align*}
			\beta(X,Y)-\beta(Y,X)
			&=\nablaf_Ydf(X)-\nablaf_Xdf(Y)-df(\nabla_YX-\nabla_XY)\\
			&=df\bigl([Y,X]-\nabla_YX+\nabla_XY\bigr)=df\bigl(T^\nabla(X,Y)\bigr).
		\end{align*}
		By \eqref{eq:tw-torsion} and \(JZ_i=iZ_i\),
		\begin{align*}
			T^\nabla(Z_i,Z_j)&=0,&T^\nabla(Z_i,Z_{\bar j})&=2i\delta_{ij}\xi,\\
			T^\nabla(\xi,Z_i)&=\mathsf AZ_i,&T^\nabla(\xi,Z_{\bar i})&=\mathsf AZ_{\bar i},
		\end{align*}
		and projection onto \(T^{1,0}N\) gives \eqref{eq:basic-commutators}. For
		a section of \(T^*M\otimes f^*TN\), the curvature of the induced
		connection is
		\[
		\bigl(R(X,Y)df\bigr)(Z)=R^N(dfX,dfY)dfZ-df\bigl(R^\nabla(X,Y)Z\bigr),
		\]
		and the alternation of the second covariant derivative is
		\[
		\nabla^2_{X,Y}-\nabla^2_{Y,X}=R(X,Y)-\nabla_{T^\nabla(X,Y)} .
		\]
		Taking \(X=Z_C\), \(Y=Z_B\), evaluating on \(Z_A\) and using
		\(T^\nabla(Z_C,Z_B)=-T^D_{BC}Z_D\) gives \eqref{eq:ricci-identity}; in
		particular the torsion term equals \(2i\delta_{ij}f_{A0}\) for
		\((B,C)=(i,\bar j)\).
	\end{proof}
	
	\section{Localization}\label{app:localization}
	
	Fix \(\vartheta(s)=e^{-1/s}\) for \(s>0\), \(\vartheta(s)=0\) for
	\(s\le0\), and put
	\[
	\zeta(s)=\frac{\vartheta(3/2-s)}{\vartheta(3/2-s)+\vartheta(s-1)},\qquad
	\eta=\zeta^4,\qquad M_1=\|\zeta'\|_\infty,\qquad M_2=\|\zeta''\|_\infty .
	\]
	The denominator is positive, \(\zeta\) decreases from \(1\) to \(0\) on
	\([1,3/2]\) and is constant elsewhere, and
	\begin{equation}\label{eq:eta-bounds}
		\begin{aligned}
			|\eta'|^2&=16\zeta^6|\zeta'|^2\le16M_1^2\eta,\\
			\eta''&=12\zeta^2(\zeta')^2+4\zeta^3\zeta''\ge-4M_2\sqrt\eta .
		\end{aligned}
	\end{equation}
	For \(o\in M\) and \(R>0\) put \(r=d_g(o,\cdot)\) and
	\begin{equation}\label{eq:common-cutoff}
		\psi_R=\eta(r/R)^2,\qquad
		\operatorname{supp}\psi_R\subset\overline{B_{3R/2}(o)}\Subset B_{2R}(o),
		\qquad\psi_R|_{B_R(o)}=1 .
	\end{equation}
	
	\begin{lemma}\label{lem:scalar-localization}
		Let \((M^\ell,g)\) be complete with \(\Ric^g\ge-K_gg\), and let \(w\ge0\)
		be smooth with
		\begin{equation}\label{eq:scalar-inequality}
			\tfrac12\Delta w\ge Aw^2-b_0w-D_0\qquad\text{on }\{w>T_0\},
		\end{equation}
		where \(A>0\) and \(b_0,D_0,T_0\ge0\). Then
		\begin{equation}\label{eq:scalar-bound}
			w\le\max\Bigl\{T_0,\frac{b_0+\sqrt{b_0^2+4AD_0}}{2A}\Bigr\}.
		\end{equation}
		If \eqref{eq:scalar-inequality} holds on \(B_{2R}(o)\) only, then
		\eqref{eq:scalar-bound} holds on \(B_R(o)\) with \(b_0\) replaced by
		\(b_0+c_\ell(R^{-2}+\sqrt{K_g}R^{-1})\), where \(c_\ell\) depends only
		on \(\ell\), \(M_1\) and \(M_2\).
	\end{lemma}
	
	\begin{proof}
		Write \(\psi=\psi_R\). With the derivatives of \(\eta\) evaluated at
		\(r/R\),
		\[
		d\psi=\frac{2\eta\eta'}R\,dr,\qquad
		\Delta\psi=\frac{2(\eta')^2+2\eta\eta''}{R^2}+\frac{2\eta\eta'}R\,\Delta r .
		\]
		Off the cut locus, Laplacian comparison gives
		\(\Delta r\le(\ell-1)/r+\sqrt{(\ell-1)K_g}\). Since \(\eta'\le0\) and
		\(\eta'=0\) for \(r<R\), \eqref{eq:eta-bounds} gives
		\begin{align}
			\frac{|d\psi|^2}\psi&=\frac{4(\eta')^2}{R^2}\le\frac{64M_1^2}{R^2},
			\notag\\
			-\Delta\psi&\le-\frac{2\eta\eta''}{R^2}
			+\frac{2\eta|\eta'|}R\Bigl(\frac{\ell-1}R+\sqrt{(\ell-1)K_g}\Bigr)
			\notag\\
			&\le\frac{8M_2}{R^2}+\frac{8M_1}R\Bigl(\frac{\ell-1}R
			+\sqrt{(\ell-1)K_g}\Bigr)
			\le c_\ell\Bigl(\frac1{R^2}+\frac{\sqrt{K_g}}R\Bigr).
			\label{eq:cutoff-bounds}
		\end{align}
		Extend \(\Phi=\psi w\) by zero outside \(B_{2R}(o)\); it is continuous
		with compact support and attains its maximum at some
		\(x_R\in\overline{B_{3R/2}(o)}\). If \(\Phi(x_R)=0\), then \(w=0\) on
		\(B_R(o)\). Assume \(\Phi(x_R)>0\), so \(\psi(x_R)>0\). If
		\(w(x_R)\le T_0\), then \(\sup_{B_R}w\le\Phi(x_R)\le T_0\). Otherwise
		\(w(x_R)>T_0\), and at \(x_R\) we have \(d\Phi=0\), that is
		\(dw=-w\,d\psi/\psi\), and
		\[
		0\ge\psi\,\Delta\Phi=\psi^2\Delta w+\psi w\,\Delta\psi-2w|d\psi|^2 .
		\]
		By \eqref{eq:scalar-inequality}, \eqref{eq:cutoff-bounds} and
		\(\psi\le1\),
		\[
		0\ge2\psi^2(Aw^2-b_0w-D_0)
		-c_\ell\Bigl(\frac1{R^2}+\frac{\sqrt{K_g}}R\Bigr)\psi w
		-\frac{128M_1^2}{R^2}\psi w ,
		\]
		hence \(A\Phi^2-(b_0+\eps_R)\Phi-D_0\le0\) with
		\(\eps_R=c'_\ell(R^{-2}+\sqrt{K_g}R^{-1})\). Thus
		\[
		\sup_{B_R(o)}w\le\Phi(x_R)\le
		\frac{b_0+\eps_R+\sqrt{(b_0+\eps_R)^2+4AD_0}}{2A}.
		\]
		If \(x_R\) is a cut point of \(o\), we use Calabi's argument
		\cite{Calabi}: for a minimizing geodesic \(\gamma\) from \(o\) to
		\(x_R\), the function \(r_\eps=\eps+d_g(\gamma(\eps),\cdot)\) is smooth
		near \(x_R\), satisfies \(r_\eps\ge r\) with equality at \(x_R\), the
		corresponding cutoff \(\psi_\eps\le\psi\) touches \(\psi\) at \(x_R\), and
		the constants converge to those of \eqref{eq:cutoff-bounds} as
		\(\eps\downarrow0\). Letting \(R\to\infty\) gives \eqref{eq:scalar-bound}.
	\end{proof}
	
	\begin{lemma}\label{lem:reeb-hessian}
		Assume \(|D\xi|_{\op}\le L\), \(D_\xi\xi=0\) and
		\(\ip{R^g(X,\xi)\xi}{X}\ge-k_\xi|X|^2\) for \(X\perp\xi\), with
		\(k_\xi\ge0\). Off \(o\) and its cut locus,
		\begin{equation}\label{eq:reeb-hessian}
			\Hess r(\xi,\xi)\le\frac2r+2\ell_\xi,\qquad
			\ell_\xi=\sqrt{L^2+k_\xi/2},
		\end{equation}
		and consequently
		\begin{equation}\label{eq:r-square-reeb}
			\xi^2(r^2)\le6+4\ell_\xi r .
		\end{equation}
	\end{lemma}
	
	\begin{proof}
		Let \(\gamma\colon[0,s]\to M\) be a unit speed minimizing geodesic from
		\(o\) to \(x\), and put \(W=\xi-\ip\xi{\dot\gamma}\dot\gamma\). Then
		\(W\perp\dot\gamma\), \(|W|\le1\), and
		\(D_{\dot\gamma}W=D_{\dot\gamma}\xi-\ip{D_{\dot\gamma}\xi}{\dot\gamma}
		\dot\gamma\), so \(|D_{\dot\gamma}W|\le L\). For \(h\in C^1\) with
		\(h(0)=0\), \(h(s)=1\), the second variation formula gives
		\[
		\Hess r(W,W)\le\int_0^s\bigl(|D_{\dot\gamma}(hW)|^2
		-h^2\ip{R^g(W,\dot\gamma)\dot\gamma}{W}\bigr)dt .
		\]
		The curvature symmetries and the hypothesis give
		\[
		-h^2\ip{R^g(W,\dot\gamma)\dot\gamma}{W}
		=-h^2\ip{R^g(\dot\gamma,\xi)\xi}{\dot\gamma}\le k_\xi h^2 ,
		\]
		and \(|D_{\dot\gamma}(hW)|^2\le2(h')^2+2L^2h^2\). Therefore
		\[
		\Hess r(W,W)\le\int_0^s\bigl(2(h')^2+(2L^2+k_\xi)h^2\bigr)dt
		=2\int_0^s\bigl((h')^2+\ell_\xi^2h^2\bigr)dt .
		\]
		Take \(h(t)=\sinh(\ell_\xi t)/\sinh(\ell_\xi s)\) if \(\ell_\xi>0\) and
		\(h(t)=t/s\) if \(\ell_\xi=0\); then \(h''=\ell_\xi^2h\), and integration
		by parts gives
		\[
		\int_0^s\bigl((h')^2+\ell_\xi^2h^2\bigr)dt
		=\bigl[hh'\bigr]_0^s-\int_0^sh\,(h''-\ell_\xi^2h)\,dt=h'(s),
		\]
		with \(h'(s)=\ell_\xi\coth(\ell_\xi s)\le\ell_\xi+s^{-1}\) if
		\(\ell_\xi>0\) and \(h'(s)=s^{-1}\) if \(\ell_\xi=0\). Hence
		\(\Hess r(W,W)\le2\ell_\xi+2/s\). Finally
		\(\Hess r(\xi,\xi)=\Hess r(W,W)\) because \(\Hess r(\dot\gamma,\cdot)=0\),
		and since \(D_\xi\xi=0\),
		\[
		\xi^2(r^2)=2(\xi r)^2+2r\Hess r(\xi,\xi)
		\le2+2r\Bigl(\frac2r+2\ell_\xi\Bigr)=6+4\ell_\xi r .\qedhere
		\]
	\end{proof}
	
	For the localization of \(\calL_t\) we use the following comparison
	for the sub-Laplacian of the Webster distance on complete
	pseudo-Hermitian manifolds with \(\Ric^{\TW}\ge-KG\) and
	\(|\mathsf A|+|\nabla\mathsf A|\le C_A\) (cf.\ \cite{CDRY21}):
	\begin{equation}\label{eq:sublap-comparison}
		\Delta_br\le C\Bigl(\frac1r+1\Bigr)\quad\text{off the cut locus of }o,
		\qquad C=C(m,K,C_A).
	\end{equation}
	By Lemma~\ref{lem:source-reeb}, Lemma~\ref{lem:reeb-hessian} applies with
	\(L=1+a_0\) and \(k_\xi=(2a_0+a_0^2+a_1-1)_+\).
	
	\begin{remark}\label{rem:sublap}
		If in addition \(\Ric^g\ge-K_gg\), then for \(t\ge1\) inequality
		\eqref{eq:sublap-comparison} can be replaced by the Riemannian Laplacian
		comparison. Indeed, since \(D_\xi\xi=0\), \(\xi^2r=\Hess r(\xi,\xi)\), and
		\[
		\calL_tr=\Delta_br+t\xi^2r=\Delta r+(t-1)\Hess r(\xi,\xi)
		\le\frac{2m}r+\sqrt{2mK_g}+(t-1)\Bigl(\frac2r+2\ell_\xi\Bigr)
		\]
		by Laplacian comparison and \eqref{eq:reeb-hessian}. In
		Theorem~\ref{thm:cr-schwarz} the optimal parameter \(t=c_m/a_0\)
		satisfies \(t\ge1\) if \(a_0\le c_m\), and \(t\to\infty\) is used if
		\(a_0=0\).
	\end{remark}
	
	\begin{lemma}\label{lem:anisotropic-localization}
		Under \eqref{eq:sublap-comparison} and the hypotheses of
		Lemma~\ref{lem:reeb-hessian}, fix \(t>0\), put \(\calL_t=\Delta_b+t\xi^2\),
		and let \(c_t=c_t(m,K,a_0,a_1,t)>0\). Then
		\begin{equation}\label{eq:anisotropic-cutoff-bounds}
			\frac{\Gamma_t(\psi_R)}{\psi_R}\le\frac{c_t}{R^2},\qquad
			-\calL_t\psi_R\le c_t\Bigl(\frac1{R^2}+\frac1R\Bigr).
		\end{equation}
		If \(w\ge0\) is smooth and \(\calL_tw\ge2Aw^2-2b_0w\) on \(B_{2R}(o)\)
		with \(A>0\), \(b_0\ge0\), then
		\[
		\sup_{B_R(o)}w\le\frac{b_0+c_t(R^{-2}+R^{-1})}A ,
		\]
		and if the inequality holds on \(M\), then \(w\le b_0/A\).
	\end{lemma}
	
	\begin{proof}
		The operator \(\calL_t\) is elliptic with
		\(\Gamma_t(w)=|d_Hw|^2+t(\xi w)^2\le\max\{1,t\}|dw|^2\). By
		\eqref{eq:eta-bounds}, on \(\{\psi_R>0\}\),
		\[
		\frac{\Gamma_t(\psi_R)}{\psi_R}=\frac{4|\eta'(r/R)|^2}{R^2}\Gamma_t(r)
		\le\frac{64M_1^2\max\{1,t\}}{R^2}\,\eta(r/R)
		\le\frac{64M_1^2\max\{1,t\}}{R^2},
		\]
		and the quotient extends continuously by zero where \(\psi_R=0\). Since
		\(\calL_t(\phi\circ r)=\phi'\calL_tr+\phi''\Gamma_t(r)\),
		\begin{align*}
			\calL_t\psi_R&=\frac{2(\eta')^2+2\eta\eta''}{R^2}\Gamma_t(r)
			+\frac{2\eta\eta'}R\calL_tr,\\
			\calL_tr&=\Delta_br+t\Hess r(\xi,\xi)\le C(r^{-1}+1)
			+t\Bigl(\frac2r+2\ell_\xi\Bigr)\le c_t(r^{-1}+1),
		\end{align*}
		using \eqref{eq:sublap-comparison}, \eqref{eq:reeb-hessian} and
		\(\xi^2r=\Hess r(\xi,\xi)\). Together with \(\Gamma_t(r)\le\max\{1,t\}\)
		and \(r\ge R\) on the support of \(\eta'\), this gives
		\eqref{eq:anisotropic-cutoff-bounds}. Let \(x_R\) be a maximum point of
		\(\Phi=\psi_Rw\) with \(\Phi(x_R)>0\). There \(dw=-w\,d\psi_R/\psi_R\),
		and
		\begin{align*}
			0\ge\psi_R\calL_t\Phi&=\psi_R^2\calL_tw+\psi_Rw\calL_t\psi_R
			-2w\Gamma_t(\psi_R)\\
			&\ge2A\Phi^2-2b_0\Phi-c_t\Bigl(\frac1{R^2}+\frac1R\Bigr)\Phi ,
		\end{align*}
		after enlarging \(c_t\). Hence \(A\Phi\le b_0+c_t(R^{-2}+R^{-1})\) at the
		maximum, and the same bound holds for \(Aw\) on \(B_R(o)\). Cut points are
		handled by Calabi's argument as in Lemma~\ref{lem:scalar-localization},
		the singular term in \eqref{eq:reeb-hessian} becoming \(2/(r-\eps)\).
		Letting \(R\to\infty\) with \(t\) fixed gives the global bound.
	\end{proof}
	
	\noindent\textbf{Acknowledgements} The first author was partially supported by the Basic Research Program of Jiangsu (Grant No. BK20250730). We would also like to thank the referee for the comments and suggestions that helped improve our manuscript.
	
	\medskip
	\noindent\textbf{Data availability}
	Data sharing is not applicable to this article as no datasets were generated or analyzed during the current study.
	
	\vspace{1em}
	\noindent\textbf{Declarations}
	
	\medskip
	\noindent\textbf{Conflict of interest}
	On behalf of all authors, the corresponding author states that there is no conflict of interest.
	
	\medskip
	\noindent\textbf{Open Access}
	This article is licensed under a Creative Commons Attribution 4.0 International License, which permits use, sharing, adaptation, distribution and reproduction in any medium or format, as long as you give appropriate credit to the original author(s) and the source, provide a link to the Creative Commons licence, and indicate if changes were made. The images or other third party material in this article are included in the article's Creative Commons licence, unless indicated otherwise in a credit line to the material. If material is not included in the article's Creative Commons licence and your intended use is not permitted by statutory regulation or exceeds the permitted use, you will need to obtain permission directly from the copyright holder. To view a copy of this licence, visit \url{http://creativecommons.org/licenses/by/4.0/}.

\end{document}